\documentclass[11pt]{amsart}
\usepackage{amsthm}
\usepackage{amsfonts}
\usepackage{amsmath,amssymb}
\usepackage{indentfirst,latexsym,bm,amsthm,graphicx,colortbl}
\usepackage{fancyhdr}
\usepackage{times}
\usepackage{amsmath,amsfonts,amssymb,amsthm}
\usepackage{epsfig}
\usepackage[all]{xy}
\newtheorem{theorem}{Theorem}[section]
\newtheorem{lemm}[theorem]{Lemma}
\newtheorem{prop}[theorem]{Proposition}
\newtheorem{theo}[theorem]{Theorem}
\theoremstyle{definition}
\newtheorem{defi}[theorem]{Definition}
\newtheorem{example}[theorem]{Example}

\newtheorem{coro}[theorem]{Corollary}
\theoremstyle{remark}
\newtheorem{remark}[theorem]{Remark}

\numberwithin{equation}{section}

\def \<{\langle}
\def \>{\rangle}

\def \g{{\frak{g}}}
\def \h{{\hbar}}

\def \be{\begin{equation}\label}
	\def \ee{\end{equation}}
\def \bex{\begin{example}\label}
	\def \eex{\end{example}}
\def \bl{\begin{lem}\label}
	\def \el{\end{lem}}
\def \bt{\begin{thm}\label}
	\def \et{\end{thm}}
\def \bp{\begin{prop}\label}
	\def \ep{\end{prop}}
\def \br{\begin{rem}\label}
	\def \er{\end{rem}}
\def \bc{\begin{coro}\label}
	\def \ec{\end{coro}}
\def \bd{\begin{de}\label}
	\def \ed{\end{de}}

\begin{document}
	
	\title[Non-weight modules over generalized Heisenberg-Virasoro algebra of rank two]
	{Non-weight modules over generalized Heisenberg-Virasoro algebra of rank two}
		
		
	\author[WEN]{Yi Wen}
	\address{Department of Mathematics, Shanghai University,
		Shanghai 200444, China}\email{even-02@shu.edu.cn}
	\author[JING]{Naihuan Jing$^\star$}
	\address{Department of Mathematics, North Carolina State University,
		Raleigh, NC 27695, USA}
	\email{jing@ncsu.edu}
	\thanks{$^\star$ N. Jing, Corresponding Author}
	\author[SUN]{Jiancai Sun}
	\address{Department of Mathematics, Shanghai University, and Newtouch Center for Mathematics of Shanghai University,
		Shanghai 200444, China}\email{jcsun@shu.edu.cn}
	\author[ZHANG]{Honglian Zhang}
	\address{Department of Mathematics, Shanghai University, and Newtouch Center for Mathematics of Shanghai University,
		Shanghai 200444, China}\email{hlzhangmath@shu.edu.cn}
	
	\date{Revision: September 8, 2025}
	\begin{abstract}
		In this paper, we study a class of non-weight modules over the generalized Heisenberg-Virasoro algebra of rank two $\widetilde{L}(p_1, p_2)$.
We construct a family of irreducible
		$\widetilde{L}(p_1, p_2)$-modules, determine the isomorphism classes and show that these modules exhaust all the $\widetilde{L}(p_1, p_2)$-modules that are free modules of rank one over the Cartan subalgebra.
	\end{abstract}
	
	
	\maketitle
	
	\section{Introduction}
	Studying various types of modules has been an important problem in the development of the representation theory of Lie algebras. For certain types of
Lie algebras and Lie superalgebras with good properties, it is meaningful to classify those modules with special properties. For instance, \cite{B2} and \cite{M} classified irreducible modules of the complex Lie algebra $\mathfrak{sl}_2$. Harish-Chandra modules (irreducible weight modules with finite-dimensional weight spaces) of various Lie algebras were studied for the Virasoro algebra in \cite{BF,M2}, for generalized Virasoro algebras \cite{GLZ2,LZ,S}, for (generalized) Heisenberg-Virasoro algebras in \cite{LG,LZ2}, and for irreducible quasifinite modules of Block type Lie algebras $\mathcal B(p)$ in \cite{CGZ,S2,S3,SXX}.

	Recently, a class of non-weight modules of a Lie algebra (or superalgebra) $\mathfrak g$ have been studied, particularly those modules
for which the universal enveloping algebra $U(\mathfrak{h})$ acts freely, where $\mathfrak{h}$ is the Cartan subalgebra of $\mathfrak g$.
The $U(\mathfrak{h})$-free modules were firstly introduced as quotients of fraction modules for the Virasoro algebra in \cite{GLZ} and for the complex matrix algebra $\mathfrak{sl}_{n+1}$ in \cite{N}. Nilsson \cite{N2} showed that a finite-dimensional simple Lie algebra has nontrivial $U(\mathfrak{h})$-free modules of finite rank if and only if it is of type $A$ or $C$, and he classified the isomorphism classes of
	$U(\mathfrak{h})$-free modules of rank $1$ in type $C$. In \cite{TZ2}, a class of modules $\Omega(\bm{\lambda},\alpha)=\mathbb{C}[\partial_1,\partial_2,...,\partial_n]$ was defined for the Witt
algebra $\mathcal{W}_n$, which turned out to be all possible $\mathcal{W}_n$-modules with free $U(\mathfrak{h})$-modules of rank 1, where $\mathfrak{h}=\mathop{\oplus}^n\limits_{i=1}\mathbb{C}\partial_i$.
Moreover, 
similar non-weight modules for many other Lie algebras and Lie superalgebras were also studied, such as the Kac-Moody algebras in \cite{CTZ}, the centerless Heisenberg-Virasoro algebras $Vir(a, b)$ in \cite{HCS}, the super-Virasoro algebra in \cite{YYX}, the $N = 2$ superconformal algebra in \cite{YYX2}, Cartan type $S$ Lie algebras in \cite{Z} and so on.\par
	The non-weight modules that are free over the Cartan subalgebra are closely related to off-shell field representations, which could be used for describing the symmetry of field theories in mathematical physics, especially in quantum field theory and conformal field theory. In quantum field theory, symmetry operations act on off-shell propagators through representations of Lie algebras.
The structure of non-weight modules is more flexible than traditional weight modules, providing extra degrees of freedom for off-shell fields.
In non-weight modules, the free action of the Cartan subalgebra corresponds to the fact that dynamics do not constrain the degrees of freedom of fields in off-shell representation. In addition, if $\mathbb{C}[\partial_1,\partial_2,...,\partial_n]$ is a non-weight module with free action over the Cartan subalgebra, then the partial derivatives $\partial_1,\partial_2,...,\partial_n$ in off-shell field representations could be regarded as the coordinates of the momentum space, which could be connected to the space-time coordinates by the Fourier transform. Fourier duality reflects the deep connection between algebraic structures and physical contexts.

	The Heisenberg-Virasoro algebra is a central extension of the Lie
	algebra of differential operators on a circle of order at most one. This algebra was generalized to rank two with central extension in \cite{XLT}.
In particular, the Heisenberg-Virasoro algebra of rank two contains a Virasoro-like subalgebra introduced by Kirkman, Procesi, and Small in \cite{KPS}. By generalizing the Witt algebra, Block introduced
a simple infinite dimensional Lie algebra in 1958 \cite{B} to study modular Lie algebras. The Block type Lie algebra $\mathcal B(p)$ is a certain algebra generalizing the original Block algebra.\par
	In \cite{GWL} and \cite{XSX}, modules with free $U(\mathfrak{h})$ action for the Block type algebra $\mathcal{B}(p)$ and more general Lie algebra $\mathcal{B}(p_1, p_2)$ are studied. In this paper, we will introduce generalized rank two (centerless) Heisenberg-Virasoro algebras $L(p_1, p_2)$ (see \eqref{e:L}), which contain $\mathcal{B}(p_1, p_2)$ as subalgebras.
Then we will study non-weight modules over the algebra $L(p_1, p_2)$ by constructing certain families of $L(p_1, p_2)$-modules. We
	 classify their isomorphism classes, moreover, we will show that they are essentially all the non-weight modules that are free over $U(\mathfrak h)$ of rank 1 under restriction.
	
	Throughout the paper, $\mathbb{Z}, \mathbb{Z}^\ast, \mathbb{N}, \mathbb{C}$ and $\mathbb{C}^{\times}$ will denote the sets of integers, non-zero integers, nonnegative integers, complex numbers, and non-zero complex numbers, respectively. All vector spaces and algebras are assumed to be over $\mathbb{C}$. This paper is organized as
	follows: In Section 2, we introduce some basic definitions and results and give the module structure of the $U(\mathfrak{h})$-free modules of $\widetilde{L}(p_1, p_2)$ (see \eqref{2.13}). In addition, we determine the isomorphism classes and irreducibilities of these modules. In Section 3, we show that these modules exhaust all $U(\mathfrak{h})$-free $\widetilde{L}(p_1, p_2)$-modules of rank 1.\par
	
\section{The algebra $\widetilde{L}(p_1, p_2)$ and the module $\Omega(\bf{\lambda}, \bf{\alpha})$}

\begin{defi}
		For fixed $\textbf{p}\in\mathbb C^2$, the {\it generalized Heisenberg-Virasoro algebra $L(p_1,p_2)$ of rank two without center} is the Lie algebra over $\mathbb{C}$ spanned by $t^{\textbf{m}}, E(\textbf{m})$ subject to the relations:
		\begin{equation}\label{e:L}
			\begin{aligned}
				\, [t^{\textbf{m}_1}, t^{\textbf{m}_2}]&=0, \\
				[t^{\textbf{m}_1},E(\textbf{m}_2)]&=-|\textbf m_1+\textbf p, \textbf m_2+\textbf p|
				t^{\textbf{m}_1+\textbf{m}_2},\\
				[E(\textbf{m}_1),E(\textbf{m}_2)]&=
				-|\textbf m_1+\textbf p, \textbf m_2+\textbf p| E(\textbf{m}_1+\textbf{m}_2),
			\end{aligned}
		\end{equation}
	where $\textbf{m}_i=(m_{i1}, m_{i2})\in\mathbb Z^2\in\mathbb{C}^2$ and $|\textbf m_1+\textbf p, \textbf m_2+\textbf p|=(m_{11}+p_1)(m_{22}+p_2)-(m_{12}+p_2)(m_{21}+p_1) $
	is the determinant of the $2\times 2$-matrix with row vectors $\textbf m_1+\textbf p$ and $\textbf m_2+\textbf p$. The Block type algebra $\mathcal B(p_1, p_2)$ spanned by the skew derivations $E(\textbf{m})$ is a subalgebra of $L(p_1,p_2)$.
	\end{defi}
Next let's take the degree derivations $\partial_1,\partial_2$ into account:\par
Let $\widetilde L(p_1, p_2)=<L(p_1, p_2), \mathbb{C}\partial_1, \mathbb{C}\partial_2>$, where $\partial_1,\partial_2$ are the degree derivations defined by
\begin{equation}
	\begin{aligned}\label{2.13}  
		[\partial_i,E(\textbf{m})]&=m_iE(\textbf{m}),  i=1, 2, \\
		[\partial_i,t^\textbf{m}]&=(m_i+p_i)t^\textbf{m}, i=1, 2, \\
		[\partial_1,\partial_2]&=0,
	\end{aligned}
\end{equation}
then $\widetilde{L}(p_1, p_2)$ is a Lie algebra. Similarly, $\widetilde{\mathcal B}(p_1, p_2)=<\mathcal{B}(p_1, p_2), \partial_1, \partial_2>$ is a Lie subalgebra of $\widetilde L(p_1, p_2)$.
\begin{remark} \label{label4} The Cartan subalgebra of $\widetilde{L}(p_1, p_2)$ is $\mathbb{C}$-span of $\partial_1, \partial_2$, and $E(\textbf{0})$, also it is clear that

	(1) When $p_1=p_2=0$, the algebra $L(0,0)$ is isomorphic to the {\it Heisenberg-Virasoro algebra	of rank two without centers} \cite{LTW,TWX}; $\widetilde{\mathcal B}(0,0)$ is isomorphic to the \emph{Virasoro-like algebra} $\mathcal{L}$ in \cite{WT}.
	
	(2) For fixed $\textbf p=(p_1,p_2)\in\mathbb C^2$, the algebra $L(p_1, p_2)$ can be realized as the Lie algebra of
	generalized power elements $t^{\textbf m}=t_1^{m_1+p_1}t_2^{m_2+p_2}$ (${\textbf m}\in\mathbb Z^2$)  and skew derivations $E(\textbf{m})=t_1^{m_1}t_2^{m_2}[(m_2+p_2)(\partial_1+p_1)-(m_1+p_1)(\partial_2+p_2)]$ (${\textbf m}\in\mathbb Z^2$). In particular, $E(\textbf{0})=p_2\partial_1-p_1\partial_2$, so the Cartan subalgebra of $\widetilde{L}(p_1, p_2)$ is $\mathfrak{h}=\mathbb{C}\partial_1\oplus\mathbb{C}\partial_2$.\par
	
	(3) There exists a degree derivation $D$ of $\widetilde{L}(p,q)$, it satisfies that $D(E(\textbf{m}))=0, D(t^{\textbf{m}})=t^{\textbf{m}}$ for all $\textbf{m}\in\mathbb{Z}^2$, and $[D(\partial_i),t^{\textbf{m}}]=p_it^{\textbf{m}}, [D(\partial_i),E(\textbf{m})]=0,i=1,2$, especially, $[\partial_i-D(\partial_i),t^{\textbf{m}}]=m_it^{\textbf{m}}, [\partial_i-D(\partial_i),E(\textbf{m})]=m_iE(\textbf{m}),i=1,2$. But $D$ is an outer derivation and does not belong to the algebra $\widetilde{L}(p_1, p_2)$.
\end{remark}
Let $\mathbb{C}[\partial_1,\partial_2]$ be the polynomial algebra in variables $\partial_1,\partial_2$. We now define an action of $L(p_1, p_2)$ on the space $\mathbb{C}[\partial_1,\partial_2]$ as follows.

Fix  $\bm{\lambda}=(\lambda_1,\lambda_2)\in\mathbb{C}^{\times}\times\mathbb{C}^{\times},\alpha\in\mathbb{C}$ and $b_{\textbf{0}}\in\mathbb{C}^{\times}$.
For any 
polynomial $f(\partial_1,\partial_2)=\sum\limits_{k=0}^M\sum\limits_{l=0}^N a_{kl}\partial_1^k \partial_2^l\in\mathbb{C}[\partial_1,\partial_2]$, we define
\begin{equation}\label{2.19}
	E(\textbf{m})f(\partial_1,\partial_2)=\lambda_1^{m_1}\lambda_2^{m_2}[(m_2+p_2)(\partial_1+p_1\alpha)-(m_1+p_1)(\partial_2+p_2\alpha)]f(\partial_1-m_1,\partial_2-m_2),
\end{equation}
\begin{equation}\label{2.20}
	t^\textbf{m}f(\partial_1,\partial_2)=\lambda_1^{m_1}\lambda_2^{m_2}b_{\textbf{0}}f(\partial_1-m_1-p_1,\partial_2-m_2-p_2),
\end{equation}
where $f(\partial_1-m_1,\partial_2-m_2)=\sum\limits_{k=0}^M\sum\limits_{l=0}^N a_{kl}(\partial_1-m_1)^k (\partial_2-m_2)^l$, the actions of $\partial_1,\partial_2$ are the usual left multiplication.
It is readily seen that this action gives $\mathbb C[\partial_1, \partial_2]$ an $\widetilde{L}(p_1, p_2)$-module structure, which will be denoted as $\Omega(\bm\lambda, \alpha,b_{\textbf{0}})$, where $\bm{\lambda}\in\mathbb{C}^{\times}\times\mathbb{C}^{\times},\alpha\in\mathbb{C}$ and $b_{\textbf{0}}\in\mathbb{C}^{\times}$.\par
When $p_1=p_2=0$, the action of $L(0,0)$ on $\Omega(\bm\lambda, \alpha,b_{\textbf{0}})$ is independent from
$\alpha$. For any $\bm{\beta}\in\mathbb{C}\times\mathbb{C}$ and $b_{\textbf{0}},k\in\mathbb{C}$, we define another $L(0,0)$-module structure $\Omega(\bm{\lambda},\bm{\beta},b_{\textbf{0}},k)$ as follows.
\begin{equation}\label{4.23}
		E(\textbf{m})f(\partial_1,\partial_2)=\lambda_1^{m_1}\lambda_2^{m_2}(m_2(\partial_1+\beta_1)-m_1(\partial_2+\beta_2))f(\partial_1-m_1,\partial_2-m_2),
\end{equation}
\begin{equation}\label{equ2.6}
	t^\textbf{m}f(\bm{\partial})=
	\left\{\begin{array}{l}\lambda_1^{m_1}\lambda_2^{m_2}kf(\bm{\partial}-\textbf{m}),~~\mbox{if}~~\textbf{m}\neq\textbf{0},
		\\b_{\textbf{0}}f(\bm{\partial}),~~\mbox{if}~~\textbf{m}=\textbf{0}.
	\end{array}\right.
\end{equation}

In the following, we will discuss irreducibility of these modules depending on $\textbf{p}=(0,0)$ or not.
\subsection{The case $\textbf{p}=(p_1, p_2)\neq(0,0)$}
\
\newline
\indent
First on the module $ \Omega(\bm{\lambda},\alpha,b_{\textbf{0}})$ we define
\begin{center} $\Omega'(\bm{\lambda},\alpha,b_{\textbf{0}})=(\partial_1+p_1\alpha)\Omega(\bm{\lambda},\alpha,b_{\textbf{0}})+(\partial_2+p_2\alpha)\Omega(\bm{\lambda},\alpha,b_{\textbf{0}}),$
\end{center}
it is clear that $\Omega'(\bm{\lambda},\alpha,b_{\textbf{0}})$ is an $L(p_1, p_2)$-submodule as well as an $\widetilde{L}(p_1, p_2)$-submodule.\par
Next we show that $\Omega(\bm{\lambda},\alpha,b_{\textbf{0}})$ is simple as both $L(p_1, p_2)$ and $\widetilde{L}(p_1, p_2)$-module, and then study isomorphisms of these modules. If $\Omega(\bm{\lambda},\alpha,b_{\textbf{0}})$ and $\Omega'(\bm{\lambda},\alpha,b_{\textbf{0}})$ are $\mathcal{B}(p_1, p_2)$-module or $\widetilde{\mathcal B}(p_1, p_2)$-module, then they could be briefly written as $\Omega(\bm{\lambda},\alpha)$ and $\Omega'(\bm{\lambda},\alpha)$. We recall
 \cite[Theorem 2.3-2.4]{GWL} and \cite[Theorem 2.6-2.7]{XSX}, and organize them into the following two statements: 
\begin{theo}\label{th2.5} Let $\mathcal{B}$ denote either $\mathcal{B}(p_1, p_2)$ or $\widetilde{\mathcal B}(p_1, p_2)$, $(p_1, p_2)\in\mathbb{C}^{\times}\times\mathbb{C}^{\times}$. 
Then $\Omega'(\bm{\lambda},\alpha)$ is the unique nonzero proper $\mathcal{B}$-submodule of $\Omega(\bm{\lambda},\alpha)$. In particular, $\Omega'(\bm{\lambda},\alpha)$ is a simple $\mathcal{B}$-submodule.
\end{theo}\par
\begin{theo}\label{theo2.4}
	For $\bm{\lambda},\bm{\mu}, \textbf{p}\in\mathbb{C}^{\times}\times\mathbb{C}^{\times},\alpha,\gamma\in\mathbb{C}$, we have\par
	$(1)$ $\Omega(\bm{\lambda},\alpha)\cong\Omega(\bm{\mu},\gamma)$ as $\mathcal{B}(p_1, p_2)$-modules $\Leftrightarrow$ $\bm{\lambda}=\bm{\mu}$ and the only isomorphisms between $\Omega(\bm{\lambda},\alpha)$ and $\Omega(\bm{\mu},\gamma)$ are multiples of $\phi_{\bm{\lambda},\alpha,\gamma}:\Omega(\bm{\lambda},\alpha)\to\Omega(\bm{\lambda},\gamma),(\partial_1+p_1\alpha)^i(\partial_2+p_2\alpha)^j\mapsto(\partial_1+p_1\gamma)^i(\partial_2+p_2\gamma)^j$.\par
	$(2)$ $\Omega(\bm{\lambda},\alpha)\cong\Omega(\bm{\mu},\gamma)$ as $\widetilde{\mathcal B}(p_1, p_2)$-modules $\Leftrightarrow$ $\bm{\lambda}=\bm{\mu},\alpha=\gamma$ and the only isomorphisms between $\Omega(\bm{\lambda},\alpha)$ and $\Omega(\bm{\mu},\gamma)$ are multiples of the identity $id$.
\end{theo}
Since $\mathcal{B}(p_1, p_2)$ is a subalgebra of $L(p_1, p_2)$ and $\widetilde{\mathcal B}(p_1, p_2)$ is a subalgebra of $\widetilde{L}(p_1, p_2)$,
we can extend the above results. 
\begin{theo}
	$\Omega(\bm{\lambda},\alpha,b_{\textbf{0}})$ is a simple $L(p_1, p_2)$ or $\widetilde{L}(p_1, p_2)$-module for $(p_1, p_2)\in\mathbb{C}^{\times}\times\mathbb{C}^{\times}$.
\end{theo}
\begin{proof}
	We take $f(\partial_1,\partial_2)=(\partial_1+p_1\alpha)f_1(\partial_1,\partial_2)+(\partial_2+p_2\alpha)f_2(\partial_1,\partial_2)\in\Omega'(\bm{\lambda},\alpha,b_{\textbf{0}})$, where $f_1(\partial_1,\partial_2)\in\Omega(\bm{\lambda},\alpha,b_{\textbf{0}}), f_2(\partial_1,\partial_2)\in\Omega(\bm{\lambda},\alpha,b_{\textbf{0}})$. Then we examine if $\Omega'(\bm{\lambda},\alpha,b_{\textbf{0}})$ contains $t^{\textbf{m}}f(\partial_1,\partial_2)$. \par
	By (\ref{2.20}), we can calculate $t^{\textbf{m}}f(\partial_1,\partial_2)$ as follow:
	\begin{center}
		$\begin{aligned}
			t^{\textbf{m}}f(\partial_1,\partial_2)=&\lambda_1^{m_1}\lambda_2^{m_2}b_{\textbf{0}}[(\partial_1+p_1\alpha-m_1-p_1)f_1(\partial_1-m_1-p_1,\partial_2-m_2-p_2)
			\\&+(\partial_2+p_2\alpha-m_2-p_2)f_2(\partial_1-m_1-p_1,\partial_2-m_2-p_2)]\\
			=&\lambda_1^{m_1}\lambda_2^{m_2}b_{\textbf{0}}[(\partial_1+p_1\alpha)f_1(\partial_1-m_1-p_1,\partial_2-m_2-p_2)
			\\&+(\partial_2+p_2\alpha)f_2(\partial_1-m_1-p_1,\partial_2-m_2-p_2)\\&-(m_1+p_1)f_1(\partial_1+p_1\alpha-m_1-p_1-p_1\alpha,\partial_2+p_2\alpha-m_2-p_2-p_2\alpha)
			\\&-(m_2+p_2)f_2(\partial_1+p_1\alpha-m_1-p_1-p_1\alpha,\partial_2+p_2\alpha-m_2-p_2-p_2\alpha)]\\
			=&\lambda_1^{m_1}\lambda_2^{m_2}b_{\textbf{0}}[(\partial_1+p_1\alpha)f_1+(\partial_2+p_2\alpha)f_2+\Delta_1+\Delta_2+C_1+C_2]\\
			\equiv&\lambda_1^{m_1}\lambda_2^{m_2}b_{\textbf{0}}(C_1+C_2)\cdot 1~(\mbox{mod}~\Omega'(\bm{\lambda},\alpha,b_{\textbf{0}})),
		\end{aligned}$
	\end{center}
	where $C_1,C_2$ are the sums of terms $(m_1+p_1+p_1\alpha)^i(m_2+p_2+p_2\alpha)^j(i,j\in\mathbb{N})$ in $-(m_1+p_1)f_1,-(m_2+p_2)f_2$, respectively, and $\Delta_1=-(m_1+p_1)f_1-C_1\in\Omega'(\bm{\lambda},\alpha,b_{\textbf{0}})$, $\Delta_2=-(m_2+p_2)f_2-C_2\in\Omega'(\bm{\lambda},\alpha,b_{\textbf{0}})$. \par
	Since 1 $\notin\Omega'(\bm{\lambda},\alpha,b_{\textbf{0}})$, $t^{\textbf{m}}\Omega'(\bm{\lambda},\alpha,b_{\textbf{0}})\not\subseteq\Omega'(\bm{\lambda},\alpha,b_{\textbf{0}})$, thus $\Omega'(\bm{\lambda},\alpha,b_{\textbf{0}})$ is not an $L(p_1, p_2)$-submodule and an $\widetilde{L}(p_1, p_2)$-submodule of $\Omega(\bm{\lambda},\alpha,b_{\textbf{0}})$. By Theorem \ref{th2.5}, $\Omega'(\bm{\lambda},\alpha,b_{\textbf{0}})$ is the unique nonzero proper $\mathcal{B}$-submodule of $\Omega(\bm{\lambda},\alpha,b_{\textbf{0}})$. Hence, $\Omega(\bm{\lambda},\alpha,b_{\textbf{0}})$ is a simple $L(p_1, p_2)$-module and a simple $\widetilde{L}(p_1, p_2)$-module.
\end{proof}
\begin{theo}
	$(1)$ $\Omega(\bm{\lambda},\alpha,b_{\textbf{0}})\cong\Omega(\bm{\mu},\gamma,b_{\textbf{0}}')$ as $L(p_1, p_2)$-modules $\Leftrightarrow$ $\bm{\lambda}=\bm{\mu},b_{\textbf{0}}=b_{\textbf{0}}'$ and the only isomorphisms between $\Omega(\bm{\lambda},\alpha,b_{\textbf{0}})$ and $\Omega(\bm{\mu},\gamma,b_{\textbf{0}}')$ are multiples of $\phi_{\bm{\lambda},\alpha,\gamma}:\Omega(\bm{\lambda},\alpha,b_{\textbf{0}})\to\Omega(\bm{\mu},\gamma,b_{\textbf{0}}'),(\partial_1+p_1\alpha)^i(\partial_2+p_2\alpha)^j\mapsto(\partial_1+p_1\gamma)^i(\partial_2+p_2\gamma)^j$.\par
	$(2)$ $\Omega(\bm{\lambda},\alpha,b_{\textbf{0}})\cong\Omega(\bm{\mu},\gamma,b_{\textbf{0}}')$ as $\widetilde{L}(p_1, p_2)$-modules $\Leftrightarrow$ $\bm{\lambda}=\bm{\mu},\alpha=\gamma,b_{\textbf{0}}=b_{\textbf{0}}'$ and the only isomorphisms between $\Omega(\bm{\lambda},\alpha,b_{\textbf{0}})$ and $\Omega(\bm{\mu},\gamma,b_{\textbf{0}}')$ are multiples of the identity $id$.
\end{theo}
\begin{proof}
	$(1)$ First we show that $\phi_{\bm{\lambda},\alpha,\gamma}$ is an $L(p_1, p_2)$-module isomorphism. Note that it is a $\mathcal{B}(p_1, p_2)$-isomorphism. It follows from  (\ref{2.20}) that
	\begin{center}
		$\begin{aligned}
		&\phi_{\bm{\lambda},\alpha,\gamma}(t^{\textbf{m}}(\partial_1+p_1\alpha)^i(\partial_2+p_2\alpha)^j)\\
		=&\phi_{\bm{\lambda},\alpha,\gamma}(\lambda_1^{m_1}\lambda_2^{m_2}b_{\textbf{0}}(\partial_1+p_1\alpha-m_1-p_1)^i(\partial_2+p_2\alpha-m_2-p_2)^j)\\
		=&\lambda_1^{m_1}\lambda_2^{m_2}b_{\textbf{0}}(\partial_1+p_1\gamma-m_1-p_1)^i(\partial_2+p_2\gamma-m_2-p_2)^j\\
		=&t^{\textbf{m}}\phi_{\bm{\lambda},\alpha,\gamma}((\partial_1+p_1\alpha)^i(\partial_2+p_2\alpha)^j),
	\end{aligned}$
	\end{center}
so $\phi_{\bm{\lambda},\alpha,\gamma}$ is an $L(p_1, p_2)$-module isomorphism.\par
If $\bm{\lambda}=\bm{\mu},b_{\textbf{0}}=b_{\textbf{0}}'$, then $\Omega(\bm{\lambda},\alpha,b_{\textbf{0}})\cong\Omega(\bm{\mu},\gamma,b_{\textbf{0}}')$ as $L(p_1, p_2)$-modules. Since the only isomorphisms between $\mathcal{B}(p_1,p_2)$-modules are multiples of $\phi_{\bm{\lambda},\alpha,\gamma}$,
they are the only isomorphisms between $L(p_1,p_2)$-modules.\par
If $\Omega(\bm{\lambda},\alpha,b_{\textbf{0}})\cong\Omega(\bm{\mu},\gamma,b_{\textbf{0}}')$ as $L(p_1, p_2)$-modules, then they are isomorphic as $\mathcal{B}(p_1, p_2)$-modules. 
So multiples of $\phi_{\bm{\lambda},\alpha,\gamma}$ are the only $L(p_1, p_2)$-module isomorphisms, and $\phi_{\bm{\lambda},\alpha,\gamma}(t^{\textbf{m}}(\partial_1+p_1\alpha)^i(\partial_2+p_2\alpha)^j)=t^{\textbf{m}}\phi_{\bm{\lambda},\alpha,\gamma}((\partial_1+p_1\alpha)^i(\partial_2+p_2\alpha)^j)$ can imply $b_{\textbf{0}}=b_{\textbf{0}}'$.\par
$(2)$ is similarly proved, as the map $id$ is an $\widetilde{L}(p_1, p_2)$-module isomorphism.
\end{proof}

\subsection{The case $\textbf{p}=(p_1, p_2)=(0,0)$}
As above we define
\begin{center}
	$\Omega''(\bm{\lambda},\bm{\beta},b_{\textbf{0}},k)=(\partial_1+\beta_1)\Omega(\bm{\lambda},\bm{\beta},b_{\textbf{0}},k)+(\partial_2+\beta_2)\Omega(\bm{\lambda},\bm{\beta},b_{\textbf{0}},k),$
\end{center}
which is 
both an $L(0,0)$-submodule and an $\widetilde{L}(0,0)$-submodule.

If $\Omega(\bm{\lambda},\bm{\beta},b_{\textbf{0}},k)$ and $\Omega''(\bm{\lambda},\bm{\beta},b_{\textbf{0}},k)$ are $\mathcal{B}(0,0)$-module or $\widetilde{\mathcal B}(0,0)$-module, they can be briefly written as $\Omega(\bm{\lambda},\bm{\beta})$ and $\Omega''(\bm{\lambda},\bm{\beta})$. We recall the following results from \cite{Z}.
\begin{theo}\label{th3.13}
$\Omega''(\bm{\lambda},\bm{\beta})$ is the unique irreducible $\mathcal{L}$-submodule of $\Omega(\bm{\lambda},\bm{\beta})$, and
	\begin{center}
		$\Omega(\bm{\lambda},\bm{\beta})/\Omega''(\bm{\lambda},\bm{\beta})\cong\mathbb{C}$.
	\end{center}
\end{theo}

\begin{theo}\label{th3.14}
	The simple $\mathcal{L}$-modules $\Omega''(\bm{\lambda},\bm{\beta})\cong\Omega''(\bm{\lambda}',\bm{\beta}')$ if and only if $\bm{\lambda}=\bm{\lambda}',$ $\bm{\beta}=\bm{\beta}'$.
\end{theo}

Note that $\mathcal{L}$ is a subalgebra of $\widetilde{L}(0,0)$, we can show the following two results. 

\begin{theo}\label{th3.15}
	$\Omega(\bm{\lambda},\bm{\beta},b_{\textbf{0}},k)$ is simple as both $L(0,0)$- and $\widetilde{L}(0,0)$-module.
	\begin{proof}
		We take $f(\partial_1,\partial_2)=(\partial_1+\beta_1)f_1(\partial_1,\partial_2)+(\partial_2+\beta_2)f_2(\partial_1,\partial_2)\in\Omega''(\bm{\lambda},\bm{\beta},b_{\textbf{0}},k)$, where $f_1(\partial_1,\partial_2)\in\Omega(\bm{\lambda},\bm{\beta},b_{\textbf{0}},k), f_2(\partial_1,\partial_2)\in\Omega(\bm{\lambda},\bm{\beta},b_{\textbf{0}},k)$, then we calculate that
		\begin{center}
			$t^{\textbf{m}}f(\partial_1,\partial_2)\equiv\lambda_1^{m_1}\lambda_2^{m_2}k(C_1+C_2)\cdot 1~(\mbox{mod}~\Omega''(\bm{\lambda},\bm{\beta},b_{\textbf{0}},k)),\textbf{m}\neq\textbf{0},$
		\end{center}
	\begin{center}
		$t^{\textbf{0}}f(\partial_1,\partial_2)\equiv 0~(\mbox{mod}~\Omega''(\bm{\lambda},\bm{\beta},b_{\textbf{0}},k)),$
	\end{center}
where $C_1,C_2$ are sums of terms $(m_1+\beta_1)^i(m_2+\beta_2)^j(i,j\in\mathbb{N})$ in $-m_1f_1,-m_2f_2$, respectively, then $t^{\textbf{m}}\Omega''(\bm{\lambda},\bm{\beta},b_{\textbf{0}},k)\not\subseteq\Omega''(\bm{\lambda},\bm{\beta},b_{\textbf{0}},k)$, thus $\Omega''(\bm{\lambda},\bm{\beta},b_{\textbf{0}},k)$ is neither $L(0,0)$- nor $\widetilde{L}(0,0)$-submodule of $\Omega(\bm{\lambda},\bm{\beta},b_{\textbf{0}},k)$. But by Theorem \ref{th3.13}, $\Omega''(\bm{\lambda},\bm{\beta},b_{\textbf{0}},k)$ is the unique irreducible $\mathcal{L}$-submodule of $\Omega(\bm{\lambda},\bm{\beta},b_{\textbf{0}},k)$. Hence, $\Omega(\bm{\lambda},\bm{\beta},b_{\textbf{0}},k)$ is
simple both as an $L(0,0)$- and $\widetilde{L}(0,0)$-module.
	\end{proof}
\end{theo}

\begin{theo}\label{th3.16} 
$(1)$ $\Omega(\bm{\lambda},\bm{\beta},b_{\textbf{0}},k)\cong\Omega(\bm{\lambda}',\bm{\beta}',b_{\textbf{0}}',k')$ as $L(0,0)$-modules $\Leftrightarrow$ $\bm{\lambda}=\bm{\lambda}',b_{\textbf{0}}=b_{\textbf{0}}',k=k'$.\par
	$(2)$ $\Omega(\bm{\lambda},\bm{\beta},b_{\textbf{0}},k)\cong\Omega(\bm{\lambda}',\bm{\beta}',b_{\textbf{0}}',k')$ as $\widetilde{L}(0,0)$-modules $\Leftrightarrow$ $\bm{\lambda}=\bm{\lambda}',\bm{\beta}=\bm{\beta}',b_{\textbf{0}}=b_{\textbf{0}}',k=k'$.
\end{theo}
\begin{proof}
	$(1)$ If $\bm{\lambda}=\bm{\lambda}',b_{\textbf{0}}=b_{\textbf{0}}',k=k'$, then $\Omega(\bm{\lambda},\bm{\beta},b_{\textbf{0}},k)\cong\Omega(\bm{\lambda}',\bm{\beta}',b_{\textbf{0}}',k')$ as $L(0,0)$-modules, and the linear map $\phi_{\bm{\beta},\bm{\beta}'}:\Omega(\bm{\lambda},\bm{\beta},b_{\textbf{0}},k)\to\Omega(\bm{\lambda},\bm{\beta}',b_{\textbf{0}},k),(\partial_1+\beta_1)^i(\partial_2+\beta_2)^j\mapsto(\partial_1+\beta_1')^i(\partial_2+\beta_2')^j$ is an $L(0,0)$-module isomorphism. It is obvious that $\phi_{\bm{\beta},\bm{\beta}'}$ is a bijection. By (\ref{4.23}), we have
	\begin{center}
		$\begin{aligned}
			&\phi_{\bm{\beta},\bm{\beta}'}(E(\textbf{m})(\partial_1+\beta_1)^i(\partial_2+\beta_2)^j)\\
			=&\phi_{\bm{\beta},\bm{\beta}'}(\lambda_1^{m_1}\lambda_2^{m_2}(m_2(\partial_1+\beta_1)-m_1(\partial_2+\beta_2))(\partial_1+\beta_1-m_1)^i(\partial_2+\beta_2-m_2)^j)\\
			=&\lambda_1^{m_1}\lambda_2^{m_2}(m_2(\partial_1+\beta_1')-m_1(\partial_2+\beta_2'))(\partial_1+\beta_1'-m_1)^i(\partial_2+\alpha'_2-m_2)^j)\\
			=&E(\textbf{m})\phi_{\bm{\beta},\bm{\beta}'}((\partial_1+\beta_1)^i(\partial_2+\beta_2)^j),
		\end{aligned}$
	\end{center}
	and by (\ref{equ2.6}), we can also get $\phi_{\bm{\beta},\bm{\beta}'}(t^{\textbf{m}}(\partial_1+\beta_1)^i(\partial_2+\beta_2)^j)=t^{\textbf{m}}\phi_{\bm{\beta},\bm{\beta}'}((\partial_1+\beta_1)^i(\partial_2+\beta_2)^j)$.\par
	If $\Omega(\bm{\lambda},\bm{\beta},b_{\textbf{0}},k)\cong\Omega(\bm{\lambda}',\bm{\beta}',b_{\textbf{0}}',k')$ as $L(0,0)$-modules, there exists an isomorphism $\phi:\Omega(\bm{\lambda},\bm{\beta},b_{\textbf{0}},k)\to\Omega(\bm{\lambda},\bm{\beta}',b_{\textbf{0}},k)$ such that $\phi(E(\textbf{m})f(\partial_1,\partial_2))=E(\textbf{m})\phi(f(\partial_1,\partial_2))$ and $\phi(t^{\textbf{m}}f(\partial_1,\partial_2))$ $=t^{\textbf{m}}\phi(f(\partial_1,\partial_2))$. We consider $t^{(0,1)}$ acts on $(\partial_1+\beta_1)^i$ of $\Omega(\bm{\lambda},\bm{\beta},b_{\textbf{0}},k)$, then we get $t^{(0,1)}(\partial_1+\beta_1)^i=\lambda_2k(\partial_1+\beta_1)^i$. Apply $\phi$ to both sides of the equation to obtain $\phi(t^{(0,1)}(\partial_1+\beta_1)^i)=t^{(0,1)}\phi((\partial_1+\beta_1)^i)=\lambda_2k\phi((\partial_1+\beta_1)^i)$. We denote $\phi((\partial_1+\beta_1)^i)=f(\partial_1,\partial_2)\in\Omega(\bm{\lambda}',\bm{\beta}',b_{\textbf{0}}',k')$, then $t^{(0,1)}f(\partial_1,\partial_2)=\lambda_2kf(\partial_1,\partial_2)=\lambda_2'k'f(\partial_1,\partial_2-1)$, the coefficients of the highest degree terms imply $\lambda_2k=\lambda_2'k'$. Similarly, we consider the action of $t^{(0,2)}$ and $t^{(0,0)}$ on $(\partial_1+\beta_1)^i$, and get $\lambda_2^2k=\lambda_2'^2k'$ and $b_{\textbf{0}}=b_{\textbf{0}}'$. Then $\lambda_2=\lambda_2',k=k'$. It is similar to show $\lambda_1=\lambda_1'$. Hence, we have $\bm{\lambda}=\bm{\lambda}',b_{\textbf{0}}=b_{\textbf{0}}',k=k'$.\par
	$(2)$ If $\bm{\lambda}=\bm{\lambda}',\bm{\beta}=\bm{\beta}',b_{\textbf{0}}=b_{\textbf{0}}',k=k'$, it is obvious that $\Omega(\bm{\lambda},\bm{\beta},b_{\textbf{0}},k)\cong\Omega(\bm{\lambda}',\bm{\beta}',b_{\textbf{0}}',k')$ as $\widetilde{L}(0,0)$-modules, and $id$ is an isomorphism.\par
	If $\Omega(\bm{\lambda},\bm{\beta},b_{\textbf{0}},k)\cong\Omega(\bm{\lambda}',\bm{\beta}',b_{\textbf{0}}',k')$ as $\widetilde{L}(0,0)$-modules, then $\Omega(\bm{\lambda},\bm{\beta},b_{\textbf{0}},k)\cong\Omega(\bm{\lambda}',\bm{\beta}',b_{\textbf{0}}',k')$ as $L(0,0)$-modules, thus $\bm{\lambda}=\bm{\lambda}',b_{\textbf{0}}=b_{\textbf{0}}',k=k'$. We consider $E(0,1)$ acts on $(\partial_1+\beta_1)$ of $\Omega(\bm{\lambda},\bm{\beta},b_{\textbf{0}},k)$, i.e. $E(0,1)(\partial_1+\beta_1)=\lambda_2(\partial_1+\beta_1)^2$. Apply $\phi$ to both sides of the equation to obtain $\phi(E(0,1)(\partial_1+\beta_1))=E(0,1)\phi(\partial_1+\beta_1)=\lambda_2\phi((\partial_1+\beta_1)^2)=E(0,1)\phi(\partial_1+\beta_1)=E(0,1)(\partial_1+\beta_1)\phi(1)$. We denote $\phi(1)=g(\partial_1,\partial_2)\in\Omega(\bm{\lambda}',\bm{\beta}',b_{\textbf{0}}',k')$, then $E(0,1)(\partial_1+\beta_1)g(\partial_1,\partial_2)=\lambda_2(\partial_1+\beta_1)^2g(\partial_1,\partial_2)=\lambda_2'(\partial_1+\beta_1')(\partial_1+\beta_1)g(\partial_1,\partial_2-1).$ So we have $(\partial_1+\beta_1)g(\partial_1,\partial_2)=(\partial_1+\beta_1')g(\partial_1,\partial_2-1)$, then we see the coefficient of $\partial_1^M\partial_2^N$ in the equation and get $\beta_1=\beta_1'$, where $\partial_1^M\partial_2^N$ is the highest degree item of $g(\partial_1,\partial_2)$. It is similar to show $\beta_2=\beta_2'$. Hence, we have $\bm{\lambda}=\bm{\lambda}',\bm{\beta}=\bm{\beta}',b_{\textbf{0}}=b_{\textbf{0}}',k=k'$.
\end{proof}

\section{Classification of $\widetilde{L}(p_1, p_2)$-modules as free $U(\mathfrak{h})$-modules of rank one}
In this section, we will determine all $\widetilde{L}(p_1, p_2)$-modules as $U(\mathfrak{h})$-free modules of rank 1, where $p_1,p_2\in\mathbb{C}$, and $\mathfrak{h}=\mathbb{C}\partial_1\oplus\mathbb{C}\partial_2$ is the canonical Cartan subalgebra of $\widetilde{L}(p_1, p_2)$.\par
First, we show the following technical lemma on polynomials: nontrivial translation invariant two-variable polynomials are linear polynomials.
	\begin{lemm}\label{lemm3.0}
		Suppose $F(Y,Z)=F(Y-q_1,Z-q_2)\in\mathbb{C}[Y,Z]$ for some $q_1,q_2\in\mathbb{C}$, then we have
		\begin{center}
			$\left\{\begin{array}{l} F(Y,Z)\in\mathbb{C}[Y]~~\mbox{if}~~q_1=0,q_2\neq0,
				\\F(Y,Z)\in\mathbb{C}[Z]~~\mbox{if}~~q_1\neq0,q_2=0,
				\\F(Y,Z)=z_0+z_1Y+z_2Z~~\mbox{if}~~q_1\neq0,q_2\neq0\mbox{, where}~~z_1q_1=-z_2q_2.
			\end{array}\right.$
		\end{center}
	\end{lemm}\par
\begin{proof}
	Write $F(Y,Z)=\sum\limits_{k=0}^M\sum\limits_{l=0}^Na_{kl}Y^kZ^l$, then the assumption says that
	\begin{equation}\label{3.1.1}
		\sum\limits_{k=0}^M\sum\limits_{l=0}^Na_{kl}Y^kZ^l=\sum\limits_{k=0}^M\sum\limits_{l=0}^Na_{kl}(Y-q_1)^k(Z-q_2)^l,
	\end{equation}
and let $x,y$ be the maximum integers such that $a_{Mx}\neq0,a_{yN}\neq0$. \\
Case \textbf {(1)} $q_1=0,q_2\neq0$.\par
The equation (\ref{3.1.1}) is then written as:
	\begin{equation}\label{3.1.2}
	\sum\limits_{k=0}^M\sum\limits_{l=0}^Na_{kl}Y^kZ^l=\sum\limits_{k=0}^M\sum\limits_{l=0}^Na_{kl}Y^k(Z-q_2)^l.
    \end{equation}\par
If $N\ge1$, then the coefficient of $Y^y Z^{N-1}$ implies that $a_{y,N-1}=a_{y,N-1}-a_{yN}Nq_2$, so $a_{yN}Nq_2=0$, which contradicts with $a_{yN}\neq0,N\neq0$ and $q_2\neq0$. Thus, $N=0$ and $F(Y,Z)\in\mathbb{C}[Y]$.\\
Case \textbf{(2)} $q_1\neq0,q_2=0$.\par
This is similar to case \textbf{(1)}.\\
\textbf{(3)} $q_1\neq0,q_2\neq0$.\par
If $M>1,N\le1$, then the equation (\ref{3.1.1}) is
\begin{center}
	$\sum\limits_{k=0}^Ma_{k0}Y^k+\sum\limits_{k=0}^Ma_{k1}Y^kZ=\sum\limits_{k=0}^Ma_{k0}(Y-q_1)^k+\sum\limits_{k=0}^Ma_{k1}(Y-q_1)^k(Z-q_2),$
\end{center}
then we apply $\frac{\partial}{\partial Z}$ to both sides to get that $\sum\limits_{k=0}^Ma_{k1}Y^k=\sum\limits_{k=0}^Ma_{k1}(Y-q_1)^k$, which condradicts $M>1$. It is similar for $M\le1,N>1$. So $M>1,N>1$ or $M\le1,N\le1$.\par
If $M>1,N>1$, we could apply $\frac{\partial^{N-1}}{\partial Z^{N-1}} $ to both sides of (\ref{3.1.1}) and get
\begin{center}
	$\sum\limits_{k=0}^M(N-1)!a_{k,N-1}Y^k+\sum\limits_{k=0}^MN!a_{kN}Y^kZ=\sum\limits_{k=0}^M(N-1)!a_{k,N-1}(Y-q_1)^k+\sum\limits_{k=0}^MN!a_{kN}(Y-q_1)^k(Z-q_2),$
\end{center}
we know it is impossible.\par
Hence, we have $M\le1,N\le1$, we could denote $F(Y,Z)=z_0+z_1Y+z_2Z+z_{11}YZ$, by the equation (\ref{3.1.1}), we get $z_{11}=0$ and $z_1,z_2$ satisfy $z_1q_1=-z_2q_2$.
\end{proof}
    Let $V$ be a $\widetilde{L}(p_1, p_2)$-module such that its restriction to $U(\mathfrak{h})$ is a free module of rank 1. 
    Choose a homogeneous basis element $1\in V$, we have $V=U(\mathfrak{h})\cdot1=\mathbb{C}[\partial_1,\partial_2]\cdot1$.  For $\textbf{p}\in\mathbb C^2$ and $\bm{\partial}=(\partial_1, \partial_2)$, the polynomial $f(\bm{\partial}+\textbf{p})=f(\partial_1+p_1, \partial_2+p_2)$ is similarly defined as below \eqref{2.20}.

\begin{lemm}\label{lemm3.1}
	For any $\textbf{m}=(m_1,m_2)\in\mathbb{Z}\times\mathbb{Z}$, $\textbf{p}=(p_1, p_2)\in\mathbb{C}\times\mathbb{C}$,  $f(\bm{\partial})=f(\partial_1,\partial_2)\in\mathbb{C}[\partial_1,\partial_2]=U(\mathfrak{h})$, then $E(\textbf{m})f(\bm{\partial})\cdot1=f(\bm{\partial}-\textbf{m})E(\textbf{m})\cdot1$ and $t^{\textbf{m}}f(\bm{\partial})\cdot1=f(\bm{\partial}-\textbf{m}-\textbf{p})t^{\textbf{m}}\cdot1.$
	\begin{proof}
		Since $[\partial_1,E(\textbf{m})]=m_1E(\textbf{m}),[\partial_2,E(\textbf{m})]=m_2E(\textbf{m})$, then we can obtain $E(\textbf{m})\partial_1=(\partial_1-m_1)E(\textbf{m}),E(\textbf{m})\partial_2=(\partial_2-m_2)E(\textbf{m})$.
		Using induction on $n$, we assume that  $E(\textbf{m})\partial_1^n=(\partial_1-m_1)^nE(\textbf{m}),E(\textbf{m})\partial_2^n=(\partial_2-m_2)^nE(\textbf{m}),n\in\mathbb{N}$, then $E(\textbf{m})\partial_1^{n+1}=E(\textbf{m})\partial_1^n\partial_1=(\partial_1-m_1)^nE(\textbf{m})\partial_1=(\partial_1-m_1)^{n+1}E(\textbf{m})$. Similarly, we have $E(\textbf{m})\partial_2^{n+1}=(\partial_2-m_2)^{n+1}E(\textbf{m})$. Hence, $E(\textbf{m})\partial_1^n=(\partial_1-m_1)^nE(\textbf{m}),E(\textbf{m})\partial_2^n=(\partial_2-m_2)^nE(\textbf{m})$ for all $n\in\mathbb{N}$.\par
		Since $[\partial_1,t^{\textbf{m}}]=(m_1+p_1)t^{\textbf{m}},[\partial_2,t^{\textbf{m}}]=(m_2+p_2)t^{\textbf{m}}$, we obtain that $t^{\textbf{m}}\partial_1=(\partial_1-m_1-p_1)t^{\textbf{m}},t^{\textbf{m}}\partial_2=(\partial_2-m_2-p_2)t^{\textbf{m}}$. Using the induction, we can get $t^{\textbf{m}}\partial_1^n=(\partial_1-m_1-p_1)^nt^{\textbf{m}},t^{\textbf{m}}\partial_2^n=(\partial_2-m_2-p_2)^nt^{\textbf{m}}$ for all $n\in\mathbb{N}$.
	\end{proof}
\end{lemm}
We write the action of $E(\textbf{m}),t^{\textbf{m}}$ on $1\in V$ as:
	\begin{equation}\label{a:act-E}
		E(\textbf{m})\cdot1=g_{\textbf{m}}(\bm{\partial})\cdot1,\quad t^{\textbf{m}}\cdot1=h_{\textbf{m}}(\bm{\partial})\cdot1,
	\end{equation}
	where $g_{\textbf{m}}(\bm{\partial})$ and $h_{\textbf{m}}(\bm{\partial})$ are polynomials in $\bm{\partial}$.
	
	By \eqref{e:L} and Lemma \ref{lemm3.1}, for any $\textbf{m}, \textbf{n}\in\mathbb{Z}\times\mathbb{Z}$ and $1\in V$, we have
\begin{equation}\label{3.5}
	h_{\textbf{n}}(\bm{\partial}-\textbf{m}-\textbf{p})h_{\textbf{m}}(\bm{\partial})-h_{\textbf{m}}(\bm{\partial}-\textbf{n}-\textbf{p})h_{\textbf{n}}(\bm{\partial})=0,
\end{equation}
	\begin{equation}\label{3.6}
		h_{\textbf{n}}(\bm{\partial}-\textbf{m})g_{\textbf{m}}(\bm{\partial})-g_{\textbf{m}}(\bm{\partial}-\textbf{n}-\textbf{p})h_{\textbf{n}}(\bm{\partial})=|\textbf n+\textbf p, \textbf m+\textbf p| h_{\textbf{m}+\textbf{n}}(\bm{\partial}),
	\end{equation}
	\begin{equation}\label{3.4}
	g_{\textbf{n}}(\bm{\partial}-\textbf{m})g_{\textbf{m}}(\bm{\partial})-g_{\textbf{m}}(\bm{\partial}-\textbf{n})g_{\textbf{n}}(\bm{\partial})=|\textbf n+\textbf p, \textbf m+\textbf p| g_{\textbf{m}+\textbf{n}}(\bm{\partial}).
\end{equation}\par
In the following, we discuss the situation depending on $\textbf{p}=(0,0)$ or not. 

\subsection{The case $\textbf{p}\neq(0,0)$}
\
\newline
\indent
For simplicity, we denote
	\begin{center}
		$\textbf{m}^\bot=(m_2,-m_1),$
	\end{center}
	\begin{center}
		$\bm{\partial}=(\partial_1,\partial_2),\quad\bm{\partial}-\textbf{m}=(\partial_1-m_1,\partial_2-m_2),$
	\end{center}
	\begin{center}
		$(\bm{\partial}|\textbf{m})=m_1\partial_1+m_2\partial_2,\quad(\textbf{m}|\textbf{n})=m_1n_1+m_2n_2,$
	\end{center}
	\begin{center}
		$X_{\textbf{m}}(\bm{\partial}+\textbf{r})=m_2(\partial_1+r_1)-m_1(\partial_2+r_2),\quad\Delta_{\textbf{m}}=({\textbf{m}}|{\textbf p}^{\perp})=p_2m_1-p_1m_2,$
	\end{center}
	where $\textbf{m},\textbf{n}\in\mathbb{Z}\times\mathbb{Z},\textbf{p}, \textbf{r}\in\mathbb{C}\times\mathbb{C}$. In the following, we briefly write $X_{\textbf{m}}$  for
	$X_\textbf{m}(\bm{\partial})=m_2\partial_1-m_1\partial_2$ and the notation $X_{\textbf p}$ is similarly defined.
We will compute $h_{\textbf{n}}(\bm{\partial})$ and $g_{\textbf{n}}(\bm{\partial})$ (see \eqref{a:act-E}) depending on $\Delta_{\textbf{m}}$ being zero or not.
	
\textbf{Subcase 1}: $\Delta_{\textbf{m}}\neq0$:\par
By Theorem 3.4 of \cite{GWL} and Theorem 3.9 of \cite{XSX}, we have the following theorem:
\begin{theo}\label{th3.2}
	Suppose $\Delta_{\textbf{m}}\neq0$ and $V$ is a $\mathcal{B}(p_1, p_2)$-module or a $\widetilde{\mathcal B}(p_1, p_2)$-module and is free of rank one
when restricted to $U(\mathfrak{h})$. Then up to a parity,  $V\cong\Omega(\bm{\lambda},\alpha)$ for some $\bm{\lambda}\in\mathbb{C}^{\times}\times\mathbb{C}^{\times}$ and $\alpha\in\mathbb{C}$ with the
module structure given by:
	\begin{center}
		$E(\textbf{m})f(\bm{\partial})=\lambda_1^{m_1}\lambda_2^{m_2}(X_{\textbf{p}}-\Delta_{\textbf{m}}\alpha+X_{\textbf{m}})f(\bm{\partial}-\textbf{m}),$
	\end{center}
	and $\partial_1,\partial_2$ act as left multiplication.
\end{theo}
Now we calculate $h_{\textbf{n}}(\bm{\partial})$ and $g_{\textbf{n}}(\bm{\partial})$:
\begin{lemm}\label{a:bm}
	If $\Delta_{\textbf{m}}\neq0$, then $g_{\textbf{m}}(\bm{\partial})=g_{\textbf{m}}(X_{\textbf{m}},X_{\textbf{p}})=\lambda_1^{m_1}\lambda_2^{m_2}(X_{\textbf{p}}-\Delta_{\textbf{m}}\alpha+X_{\textbf{m}})$ and $h_{\textbf{m}}(\bm{\partial})=h_{\textbf{m}}(X_{\textbf{m}},X_{\textbf{p}})=b_{\textbf{m}}\in\mathbb{C}$.
	\begin{proof}
		Since an $L(p_1, p_2)$-module is also a $\mathcal{B}(p_1, p_2)$-module. Theorem \ref{th3.2} implies that $g_{\textbf{m}}(\bm{\partial})=\lambda_1^{m_1}\lambda_2^{m_2}(X_{\textbf{p}}-\Delta_{\textbf{m}}\alpha+X_{\textbf{m}})$ if $\Delta_{\textbf{m}}\neq0$.\par
		 If $k_1X_{\textbf{m}}+k_2X_{\textbf{p}}=\partial_1(k_1m_2+k_2p_2)-\partial_2(k_1m_1+k_2p_1)=0$, then
		 \begin{center}
		 	$\begin{pmatrix}
		 		m_1  & p_1\\
		 		m_2  & p_2
		 	\end{pmatrix}\begin{pmatrix}
		 	k_1\\
		 	k_2
	 	\end{pmatrix}=\begin{pmatrix}
	 	0\\
	 	0
 	\end{pmatrix}.$
		 \end{center}
		 When $\Delta_{\textbf{m}}=m_1p_2-m_2p_1\neq0$, $k_1=k_2=0$. Hence, $X_{\textbf{m}}$ and $X_{\textbf{p}}$ are linear independent, $\mathbb{C}[\partial_1,\partial_2]=\mathbb{C}[X_{\textbf{m}},X_{\textbf{p}}]$. So there exist $g_{\textbf{m}}(X_{\textbf{m}},X_{\textbf{p}}),h_{\textbf{m}}(X_{\textbf{m}},X_{\textbf{p}})\in\mathbb{C}[X_{\textbf{m}},X_{\textbf{p}}]$ such that $g_{\textbf{m}}(X_{\textbf{m}},X_{\textbf{p}})=g_{\textbf{m}}(\bm{\partial}),h_{\textbf{m}}(X_{\textbf{m}},X_{\textbf{p}})=h_{\textbf{m}}(\bm{\partial})$, where $X_{\textbf{m}}=X_{\textbf{m}}(\bm{\partial}),X_{\textbf{p}}=X_{\textbf{p}}(\bm{\partial})$.

Setting $\textbf{m}=\textbf{n}\neq\textbf{0}$ and $\Delta_{\textbf{m}}\neq0$ in (\ref{3.6}), then $h_{\textbf{m}}(X_{\textbf{m}},X_{\textbf{p}}-\Delta_{\textbf{m}})g_{\textbf{m}}(X_{\textbf{m}},X_{\textbf{p}})-g_{\textbf{m}}(X_{\textbf{m}}+\Delta_{\textbf{m}},X_{\textbf{p}}-\Delta_{\textbf{m}})h_{\textbf{m}}(X_{\textbf{m}},X_{\textbf{p}})=0$. Since $g_{\textbf{m}}(X_{\textbf{m}},X_{\textbf{p}})=g_{\textbf{m}}(X_{\textbf{m}}+\Delta_{\textbf{m}},X_{\textbf{p}}-\Delta_{\textbf{m}})$, we get $h_{\textbf{m}}(X_{\textbf{m}},X_{\textbf{p}})=h_{\textbf{m}}(X_{\textbf{m}},X_{\textbf{p}}-\Delta_{\textbf{m}})$. By Lemma \ref{lemm3.0}, we have $h_{\textbf{m}}(X_{\textbf{m}},X_{\textbf{p}})\in\mathbb{C}[X_{\textbf{m}}]$.

	 We set $(\textbf{m}|\textbf{n}^{\bot})=0,\Delta_{\textbf{n}}\neq0$ in (\ref{3.5}), then we get $h_{\textbf{n}}(X_{\textbf{n}}+\Delta_{\textbf{n}})h_{\textbf{m}}(X_{\textbf{m}})=h_{\textbf{m}}(X_{\textbf{m}}+\Delta_{\textbf{m}})h_{\textbf{n}}(X_{\textbf{n}})$, i.e. $[\sum\limits_{k=0}^M a_k(X_{\textbf{n}}+\Delta_{\textbf{n}})^k](\sum\limits_{l=0}^N a_l' X_{\textbf{m}}^l)=[\sum\limits_{l=0}^N a_l'(X_{\textbf{m}}+\Delta_{\textbf{m}})^l](\sum\limits_{k=0}^M a_k X_{\textbf{n}}^k)$. The coefficient of the item $X_{\textbf{n}}^{M-1} X_{\textbf{m}}^{N}$ satisfies $a_{M-1}a_N'+a_Ma_N'M\Delta_{\textbf{n}}=a_{M-1}a_N'$,
	 so  $M=0$. It is similar to get $N=0$. Hence,
	 \begin{equation}\label{3.8}
	 	h_{\textbf{m}}(X_{\textbf{m}},X_{\textbf{p}})=b_{\textbf{m}}\in\mathbb{C}~~\mbox{if}~~\Delta_{\textbf{m}}\neq0.
	 \end{equation}
 The proof is completed.
	\end{proof}
\end{lemm}
The following lemma shows that there exists $\textbf{m}\in\mathbb{Z}\times\mathbb{Z}~(\Delta_{\textbf{m}}\neq0)$ such that the constant $h_{\textbf{m}}(\bm{\partial})=b_{\textbf{m}}\neq0$ if $t^{\textbf{m}}$ acts nontrivial.
\begin{lemm}\label{lemm3.4}
	If the action constant $b_{\textbf{n}}=0$ for any $\textbf{n}\in\mathbb{Z}\times\mathbb{Z}$ such that $\Delta_{\textbf{n}}\neq0$, then $h_{\textbf{m}}(\bm{\partial})=0$
	for all $\textbf{m}\in\mathbb{Z}\times\mathbb{Z}$, i.e. the actions of all $t^{\textbf{m}}$ are trivial.
	\begin{proof}
		We assume $h_{\textbf{m}}(\bm{\partial})=b_{\textbf{n}}=0$ for any $\textbf{n}\in\mathbb{Z}\times\mathbb{Z}$ such that $\Delta_{\textbf{n}}\neq0$. When $\Delta_{\textbf{m}}=0$, we set $h_{\textbf{m}}(\bm{\partial})=\sum\limits_{k=0}^M\sum\limits_{l=0}^Na_{kl}\partial_1^k\partial_2^l$.
		
		Taking $\textbf{m},\textbf{n}$ satisfy $\Delta_{\textbf{m}}\neq0,\Delta_{\textbf{n}}\neq0,\Delta_{\textbf{m}+\textbf{n}}=0$ and $\textbf{m}+\textbf{n}\neq-2\textbf{p}$ in (\ref{3.6}), we have  $|\textbf n+\textbf p, \textbf m+\textbf p| \neq0$ and $|\textbf n+\textbf p, \textbf m+\textbf p| h_{\textbf{m}+\textbf{n}}(\bm{\partial})=0$, thus $h_{\textbf{m}+\textbf{n}}(\bm{\partial})=0$. So we get $h_{\textbf{n}}(\bm{\partial})=0$ if $\Delta_{\textbf{n}}=0$ and $\textbf{n}\neq-2\textbf{p}$. Hence, $h_{\textbf{n}}(\bm{\partial})=0$ if $\textbf{n}\neq-2\textbf{p}$.\par
		 We set $\Delta_{\textbf{m}}\neq0,\textbf{n}=-2\textbf{p}\in\mathbb{Z}\times\mathbb{Z}$ in (\ref{3.6}), then $\Delta_{\textbf{m}+\textbf{n}}\neq0$, $\textbf{m}+\textbf{n}\neq-2\textbf{p}$, and
		 \begin{equation}\label{3.9}
		 	\begin{aligned}
		 		&\sum\limits_{k=0}^M\sum\limits_{l=0}^N a_{kl}(\partial_1-m_1)^k(\partial_2-m_2)^l[(m_2+p_2)\partial_1-(m_1+p_1)\partial_2-\Delta_{\textbf{m}}\alpha]\\
		 		=&\sum\limits_{k=0}^M\sum\limits_{l=0}^N a_{kl}\partial_1^k\partial_2^l[(m_2+p_2)\partial_1-(m_1+p_1)\partial_2-\Delta_{\textbf{m}}(\alpha+1)].
		 	\end{aligned}
		 \end{equation}\par
	Next we prove that all coefficients $a_{kl}$ satisfy $a_{kl}=0$. We discuss in the following situations:\\
	$(i)~~p_1=0,p_2\neq0$.\par
	Taking $\textbf{m}=(m_1,0),m_1\neq0$ in (\ref{3.9}), we get
	\begin{equation}\label{3.10}
		\begin{aligned}
			&\sum\limits_{k=0}^M\sum\limits_{l=0}^N a_{kl}(\partial_1-m_1)^k\partial_2^l(p_2\partial_1-m_1\partial_2-m_1p_2\alpha)\\
			=&\sum\limits_{k=0}^M\sum\limits_{l=0}^Na_{kl}\partial_1^k\partial_2^l[p_2\partial_1-m_1\partial_2-m_1p_2(\alpha+1)].
		\end{aligned}
	\end{equation}

We suppose $h_{-2\textbf{p}}\neq0$, and set $x,y$ as the maximum integers such that $a_{Mx}\neq0,a_{yN}\neq0$, then the coefficients of $\partial_1^M\partial_2^x$ in (\ref{3.10}) satisfy $p_2a_{M-1,x}-p_2a_{Mx}Mm_1-m_1a_{M,x-1}-m_1p_2\alpha a_{Mx}=p_2a_{M-1,x}-m_1a_{M,x-1}-m_1p_2(\alpha+1)a_{Mx}$, so we obtain $M=1$.
Hence, the equation (\ref{3.10}) can be written as
\begin{equation}\label{3.11}
	-(m_1\partial_2+m_1p_2\alpha)\sum\limits_{l=0}^x a_{1l}\partial_2^l=p_2\sum\limits_{l=0}^N a_{0l}\partial_2^l.
\end{equation}

If $x=N$, then the coefficients of $\partial_2^{N+1}$ in (\ref{3.11}) satisfy $-m_1a_{1N}=0$, thus $a_{1N}=a_{1x}=0$, this contradicts with $a_{1x}\neq0$, so $x\le N-1$, so we have that
 $a_{1N}=0$ and $a_{0N}\neq0$. Since the coefficients of $\partial_2^N$ in (\ref{3.11}) satisfy $-m_1a_{1,N-1}=p_2a_{0N}\neq0$,  we see that $a_{1,N-1}\neq0$ and $x=N-1$.

Consider the coefficients of $\partial_2^N$ in (\ref{3.11}), we see that $-m_1a_{1,N-1}=p_2a_{0N}\neq0$. If we replace $m_1$ with $2m_1$ in (\ref{3.10}), we have $-2m_1a_{1,N-1}=p_2a_{0N}\neq0$. But this contradicts $a_{1,N-1}\neq0$ and $a_{0N}\neq0$. Hence, $h_{-2\textbf{p}}=0$.\\
$(ii)~~p_1\neq0,p_2=0$.\par
It is seen by the similar argument that $h_{-2\textbf{p}}=0$.\\
$(iii)~~p_1\neq0,p_2\neq0$.\par
We can set $\textbf{m}=(m_1,0),m_1\neq0$ and $m_1+p_1\neq0$ in (\ref{3.9}), then similarly we obtain $h_{-2\textbf{p}}=0$.

In conclusion, $h_{\textbf{m}}(\bm{\partial})=0$ for all $\textbf{m}\in\mathbb{Z}\times\mathbb{Z}$ if $h_{\textbf{n}}(\bm{\partial})=0$ for all $\textbf{n}\in\mathbb{Z}\times\mathbb{Z}$ that satisfy $\Delta_{\textbf{n}}\neq0$.
	\end{proof}
\end{lemm}
\textbf{Subcase 2}: $\Delta_{\textbf{m}}=0$:\par
In the following lemma, we show that $h_{\textbf{m}}(\bm{\partial})$ are polynomials of degree zero when $\textbf{p}=(p_1, p_2)\in(\mathbb{C}\times\mathbb{C})\backslash\left \{ (0,0) \right \}$ and $\Delta_{\textbf{m}}=0$.
\begin{lemm}
	For any $\textbf{m}\in\mathbb{Z}\times\mathbb{Z}$ such that $\Delta_{\textbf{m}}=0$, we have $h_{\textbf{m}}(\bm{\partial})=b_{\textbf{m}}\in\mathbb{C}$.
	\begin{proof}
		First, we take $\textbf{m}=-2\textbf{n}$ in (\ref{3.6}), and $\Delta_{\textbf{m}}\neq0,\Delta_{\textbf{n}}\neq0$, then we have $b_{-\textbf{n}}=\lambda_1^{-2n_1}\lambda_2^{-2n_2}b_{\textbf{n}}$.
		
			Then $\textbf{m},\textbf{n}$ satisfy $\Delta_{\textbf{m}}=0,\Delta_{\textbf{n}}\neq0$ and $h_{\textbf{n}}(\bm{\partial})=b_{\textbf{n}}\neq0$ in (\ref{3.5}), then we have $h_{\textbf{m}}(\bm{\partial})=h_{\textbf{m}}(\bm{\partial}-\textbf{n}-\textbf{p})$, this gives $h_{\textbf{m}}(X_{\textbf{n}},X_{\textbf{p}})=h_{\textbf{m}}(X_{\textbf{n}}+\Delta_{\textbf{n}},X_{\textbf{p}}-\Delta_{\textbf{n}})$. By Lemma \ref{lemm3.0}, we get $h_{\textbf{m}}(X_{\textbf{n}},X_{\textbf{p}})=z_0+z_1X_{\textbf{n}}+z_2X_{\textbf{p}}$, then we could denote $h_{\textbf{m}}(\bm{\partial})=b_{\textbf{m}}+b_{\textbf{m}}^{(1)}\partial_1+b_{\textbf{m}}^{(2)}\partial_2$. And $h_{\textbf{\textbf{m}}}(\bm{\partial})=h_{\textbf{\textbf{m}}}(\bm{\partial}-\textbf{n}-\textbf{p})$ could imply $(n_1+p_1)b_{\textbf{m}}^{(1)}+(n_2+p_2)b_{\textbf{m}}^{(2)}=0$.\par
			For the above $\textbf{n}$, we know $b_{-\textbf{n}}=\lambda_1^{-2n_1}\lambda_2^{-2n_2}b_{\textbf{n}}\neq0$ and $\Delta_{-\textbf{n}}\neq0$, then we can similarly obtain $(n_1-p_1)b_{\textbf{m}}^{(1)}+(n_2-p_2)b_{\textbf{m}}^{(2)}=0$ by replacing $\textbf{n}$ with $-\textbf{n}$. Hence, the system of linear equations
\begin{center}
	$\begin{pmatrix}
		n_1+p_1 & n_2+p_2 \\
		n_1-p_1 & n_2-p_2
	\end{pmatrix}\begin{pmatrix}
		b_{\textbf{m}}^{(1)}\\
		b_{\textbf{m}}^{(2)}
	\end{pmatrix}=\begin{pmatrix}
	0\\
	0
\end{pmatrix}$
\end{center}
holds, and by $\begin{vmatrix}
	n_1+p_1 & n_2+p_2\\
	n_1-p_1 & n_2-p_2
\end{vmatrix}=-2\Delta_{\textbf{n}}\neq0$, we have $b_{\textbf{m}}^{(1)}=b_{\textbf{m}}^{(2)}=0$. Overall,
\begin{equation}\label{3.17}
	h_{\textbf{m}}(\bm{\partial})=b_{\textbf{m}}\in\mathbb{C}~~\mbox{if}~~\Delta_{\textbf{m}}=0.
\end{equation}
This concludes the proof.
	\end{proof}
\end{lemm}
Combining the equations (\ref{3.8}) and (\ref{3.17}), we obtain that the action $t^{\textbf{m}}\cdot 1$ is a constant action (cf. \eqref{a:act-E}):
\begin{equation}\label{3.18}
	h_{\textbf{m}}(\bm{\partial})=b_{\textbf{m}}\in\mathbb{C}~~\mbox{for all}~~\textbf{m}\in\mathbb{Z}\times\mathbb{Z}.
\end{equation}\par
The next lemma gives the relationship between $b_{\textbf{m}}$ and $b_{\textbf{0}}$ for all $\textbf{m}\in\mathbb{Z}\times\mathbb{Z}$.
\begin{lemm}\label{lemm3.6}
	For any $\textbf{m}\in\mathbb{Z}\times\mathbb{Z}$, we have that $b_{\textbf{m}}=\lambda_1^{m_1}\lambda_2^{m_2}b_{\textbf{0}}$.
	\begin{proof}
		Setting $\textbf{m}+\textbf{n}=\textbf{0},\Delta_{\textbf{m}}\neq0$ in (\ref{3.6}), then $\Delta_{\textbf{n}}\neq0,(\textbf{m}|\textbf{n}^{\bot})=0$, and $-2\lambda_1^{m_1}\lambda_2^{m_2}\Delta_{\textbf{m}}b_{\textbf{n}}=-2\Delta_{\textbf{m}}b_{\textbf{0}}$, so we obtain
		\begin{center}
			$b_{\textbf{n}}=\lambda_1^{n_1}\lambda_2^{n_2}b_{\textbf{0}}$ if $\Delta_{\textbf{n}}\neq0$.
		\end{center}
		
We take $\Delta_{\textbf{m}}\neq0,\Delta_{\textbf{n}}=0$ in (\ref{3.6}), then $\Delta_{\textbf{m}+\textbf{n}}\neq0$, and $-(\Delta_{\textbf{m}}+(\textbf{m}|\textbf{n}^{\bot}))b_{\textbf{n}}=-\lambda_1^{n_1}\lambda_2^{n_2}(\Delta_{\textbf{m}}+(\textbf{m}|\textbf{n}^{\bot}))b_{\textbf{0}}$. In addition, $\Delta_{\textbf{m}}\neq0,\Delta_{\textbf{n}}=0$ and $\Delta_{\textbf{m}}+(\textbf{m}|\textbf{n}^{\bot})=|\textbf m, \textbf n+\textbf p|\neq0$ yields that $\Delta_{\textbf{n}}=0$ and $\textbf{n}\neq-\textbf{p}$, thus
\begin{center}
	$b_{\textbf{n}}=\lambda_1^{n_1}\lambda_2^{n_2}b_{\textbf{0}}$ if $\Delta_{\textbf{n}}=0$ and $\textbf{n}\neq-\textbf{p}$.
\end{center}\par
Combining the above discussions, we can get
\begin{equation}\label{3.19}
	b_{\textbf{n}}=\lambda_1^{n_1}\lambda_2^{n_2}b_{\textbf{0}}~~\mbox{if}~~\textbf{n}\neq-\textbf{p}.
\end{equation}\par
Letting $\Delta_{\textbf{m}}\neq0,\Delta_{\textbf{n}}\neq0$ and $\textbf{m}+\textbf{n}=-\textbf{p}$ in (\ref{3.6}), then we have $\Delta_{\textbf{n}}-(\textbf{m}|\textbf{n}^{\bot})-\Delta_{\textbf{m}}\neq0$, and $\lambda_1^{m_1}\lambda_2^{m_2}b_{\textbf{n}}[\Delta_{\textbf{n}}-(\textbf{m}|\textbf{n}^{\bot})-\Delta_{\textbf{m}}]=[\Delta_{\textbf{n}}-(\textbf{m}|\textbf{n}^{\bot})-\Delta_{\textbf{m}}]b_{-\textbf{p}}$. By (\ref{3.19}), we can get
\begin{equation}\label{3.20}
	b_{-\textbf{p}}=\lambda_1^{-p_1}\lambda_2^{-p_2}b_{\textbf{0}}.
\end{equation}\par
By (\ref{3.19}) and (\ref{3.20}), we obtain
\begin{equation}\label{3.21}
	b_{\textbf{m}}=\lambda_1^{m_1}\lambda_2^{m_2}b_{\textbf{0}}~~\mbox{for all}~~\textbf{m}\in\mathbb{Z}\times\mathbb{Z}.
\end{equation}
This completes the proof.
	\end{proof}
\end{lemm}
The following lemma gives the expressions of $g_{\textbf{m}}(\bm{\partial})$ when $\textbf{p}=(p_1, p_2)\in(\mathbb{C}\times\mathbb{C})\backslash\left \{ (0,0) \right \}$ and  $\Delta_{\textbf{m}}=0$.

\begin{lemm}\label{lemm3.7}
	Suppose $\Delta_{\textbf{m}}=0$, then $g_{\textbf{m}}(\bm{\partial})=\lambda_1^{m_1}\lambda_2^{m_2}(X_{\textbf{p}}+X_{\textbf{m}})$.
	\begin{proof} It's enough to consider the nontrivial action of
		$t^{\textbf{m}}$. By Lemmas \ref{lemm3.4} and \ref{lemm3.6}, we can write $b_{\textbf{m}}=\lambda_1^{m_1}\lambda_1^{m_1}b_{\textbf{0}}$
		for some constant $b_{\textbf{0}}\neq 0$.
		
	Let $g_{\textbf{m}}(\bm{\partial})=\sum\limits_{k=0}^M\sum\limits_{l=0}^Na_{kl}\partial_1^k\partial_2^l$. Taking $\textbf{n}=\textbf{0}$ and $\Delta_{\textbf{m}}=0$ in (\ref{3.6})
	implies that $g_{\textbf{m}}(\bm{\partial})=g_{\textbf{m}}(\bm{\partial}-\textbf{p})$. According to Lemma {\ref{lemm3.0}}, we have
		\begin{equation}\label{3.23}
			\left\{\begin{array}{l} g_{\textbf{m}}(\bm{\partial})\in\mathbb{C}[\partial_2]~~\mbox{if}~~\textbf{m}=(m_1,0),p_1\neq0,p_2=0,
				\\g_{\textbf{m}}(\bm{\partial})\in\mathbb{C}[\partial_1]~~\mbox{if}~~\textbf{m}=(0,m_2),p_1=0,p_2\neq0,
				\\g_{\textbf{m}}(\bm{\partial})=d_{\textbf{m}}+d_{\textbf{m}}^{(1)}\partial_1+d_{\textbf{m}}^{(2)}\partial_2~~\mbox{if}~~\Delta_{\textbf{m}}=0,p_1\neq0,p_2\neq0\mbox{, where}~~d_{\textbf{m}}^{(1)}p_1=-d_{\textbf{m}}^{(2)}p_2.
			\end{array}\right.
		\end{equation}
	\textbf{(1)} $p_1\neq0,p_2=0$.\par
	In this case, we have $g_{(m_1,0)}(\bm{\partial})=\sum\limits_{l=0}^Na_{l}\partial_2^l$, where $a_{N}\neq0$. Taking $\textbf{m}=(m_1,0),\textbf{n}=(n_1,n_2)$ and $n_2\neq0$ in (\ref{3.6}), then $\Delta_{\textbf{m}}=0,\Delta_{\textbf{n}}\neq0,\Delta_{\textbf{m}+\textbf{n}}\neq0$, and
		\begin{equation}\label{eq3.20}
			b_{\textbf{n}}[\sum\limits_{l=0}^Na_{l}\partial_2^l-\sum\limits_{l=0}^Na_{l}(\partial_2-n_2)^l]=-b_{\textbf{m}+\textbf{n}}n_2(m_1+p_1).
		\end{equation}\par
	When $m_1+p_1=0$,
	we have $\sum\limits_{l=0}^Na_l\partial_2^l=\sum\limits_{l=0}^Na_l(\partial_2-n_2)^l$, then $N=0$. We can denote $g_{(-p_1,0)}(\bm{\partial})=d_{(-p_1,0)}\in\mathbb{C}$. When $m_1+p_1\neq0$, we discuss the coefficient equations in (\ref{eq3.20}): If $N>1$, the coefficient of $\partial_2^{N-1}$ satisfies $b_{\textbf{n}}a_NNn_2=0$, it is a contradiction; if $N=0$, the constant term satisfies $-n_2(m_1+p_1)b_{\textbf{m}+\textbf{n}}=0$, it is also a contradiction. Hence, we have $N=1$, and the constant term satisfies $b_{\textbf{n}}a_1n_2=-n_2(m_1+p_1)b_{\textbf{m}+\textbf{n}}$, i.e. $a_1=-(m_1+p_1)\lambda_1^{m_1}$. So we can get $g_{(m_1,0)}(\bm{\partial})=a_0-(m_1+p_1)\lambda_1^{m_1}\partial_2$ if $m_1+p_1\neq0$, where $a_0$ depends on $(m_1,0)$, thus we denote $a_0$ as $d_{(m_1,0)}$ to be compatible with the expression of $g_{(m_1,0)}(\bm{\partial})$.\par
	Combining the above discussion, we obtain the expressions of $g_{(m_1,0)}(\bm{\partial})$ when $p_1\neq0,p_2=0$:
	\begin{equation}\label{3.24}
		g_{(m_1,0)}(\bm{\partial})=d_{(m_1,0)}-(m_1+p_1)\lambda_1^{m_1}\partial_2~~\mbox{for all}~~ m_1\in\mathbb{Z}.
	\end{equation}\par
Next we calculate $d_{(m_1,0)}$. We set $\textbf{n}=(n_1,0),\textbf{m}=(m_1,m_2),m_2\neq0$ in (\ref{3.4}), then we have $\Delta_{\textbf{n}}=0,\Delta_{\textbf{m}}\neq0,\Delta_{\textbf{m}+\textbf{n}}\neq0$, and $n_1m_2d_{(n_1,0)}=0$,
	   this yields $d_{(n_1,0)}=0$ if $n_1\neq0$.\par
	   Hence, in the case $p_1\neq0,p_2=0$, we can get
	   \begin{equation}\label{3.25}
	   	   g_{(m_1,0)}(\bm{\partial})=\delta_{m_1,0}d_{(0,0)}(p_1,0)-(m_1+p_1)\lambda_1^{m_1}\partial_2~~\mbox{for}~~p_1\neq0,p_2=0,m_1\in\mathbb{Z},
	   \end{equation}
   where $d_{(0,0)}(p_1,0)$ is a complex number related to $(p_1,0)$.\\
\textbf{(2)} $p_1=0,p_2\neq0$.\par
The case is similar to the case \textbf{(1)}, we can obtain
\begin{equation}\label{3.26}
	g_{(0,m_2)}(\bm{\partial})=\delta_{m_2,0}d_{(0,0)}(0,p_2)+(m_2+p_2)\lambda_2^{m_2}\partial_1~~\mbox{for}~~p_1=0,p_2\neq0,m_2\in\mathbb{Z}.
\end{equation}
\textbf{(3)} $p_1\neq0,p_2\neq0$.\par
We set $\Delta_{\textbf{m}}=0,\Delta_{\textbf{n}}\neq0$ in (\ref{3.6}), then $\Delta_{\textbf{m}+\textbf{n}}\neq0$, and
 $(n_1+p_1)d_{\textbf{m}}^{(1)}+(n_2+p_2)d_{\textbf{m}}^{(2)}=\lambda_1^{m_1}\lambda_2^{m_2}[\Delta_{\textbf{n}}-(\textbf{m}|\textbf{n}^{\bot})]$. Since $d_{\textbf{m}}^{(1)}p_1=-d_{\textbf{m}}^{(2)}p_2$, then $n_1d_{\textbf{m}}^{(1)}+n_2d_{\textbf{m}}^{(2)}=\lambda_1^{m_1}\lambda_2^{m_2}[\Delta_{\textbf{n}}-(\textbf{m}|\textbf{n}^{\bot})]$, and we can get
\begin{center}
	$\left\{\begin{array}{l}d_{\textbf{m}}^{(1)}=\lambda_1^{m_1}\lambda_2^{m_2}p_2[1-\dfrac{(\textbf{m}|\textbf{n}^{\bot})}{\Delta_{\textbf{n}}}]=\lambda_1^{m_1}\lambda_2^{m_2}p_2(1+K),
		\\d_{\textbf{m}}^{(2)}=-\lambda_1^{m_1}\lambda_2^{m_2}p_1[1-\dfrac{(\textbf{m}|\textbf{n}^{\bot})}{\Delta_{\textbf{n}}}]=-\lambda_1^{m_1}\lambda_2^{m_2}p_1(1+K),
	\end{array}\right.$
\end{center}
where $K=-\dfrac{(\textbf{m}|\textbf{n}^{\bot})}{\Delta_{\textbf{n}}},\textbf{m}=K(p_1, p_2)$.\par
Now, the expressions of $g_{\textbf{m}}(\bm{\partial})$ could be written as $g_{\textbf{m}}(\bm{\partial})=d_{\textbf{m}}+\lambda_1^{m_1}\lambda_2^{m_2}(K+1)(p_2\partial_1-p_1\partial_2)=d_{\textbf{m}}+\lambda_1^{m_1}\lambda_2^{m_2}(X_{\textbf{p}}+X_{\textbf{m}})$ if $\Delta_{\textbf{m}}=0$. Then taking $\Delta_{\textbf{n}}=0,\Delta_{\textbf{m}}\neq0$ in (\ref{3.4}), we have $\Delta_{\textbf{m}+\textbf{n}}\neq0$ and $(\textbf{m}|\textbf{n}^{\bot})d_{\textbf{n}}=\lambda_1^{n_1}\lambda_2^{n_2}\Delta_{\textbf{n}}(m_2\partial_1-m_1\partial_2)=0$. If $\Delta_{\textbf{n}}=0,\Delta_{\textbf{m}}\neq0$ and $(\textbf{m}|\textbf{n}^{\bot})\neq0$, we have $\Delta_{\textbf{n}}=0$ and $\textbf{n}\neq\textbf{0}$, so we can get $d_{\textbf{n}}=0$ if $\Delta_{\textbf{n}}=0,\textbf{n}\neq\textbf{0}$.

Hence, we get
\begin{equation}\label{3.27}
	g_{\textbf{m}}(\bm{\partial})=\delta_{\textbf{m},\textbf{0}}d_{\textbf{0}}(p_1,p_2)+\lambda_1^{m_1}\lambda_2^{m_2}(X_{\textbf{p}}+X_{\textbf{m}})~~\mbox{for}~~p_1\neq0,p_2\neq0,\Delta_{\textbf{m}}=0.
\end{equation}

Combining (\ref{3.25})-(\ref{3.27}), we have
\begin{equation}\label{3.28}
	g_{\textbf{m}}(\bm{\partial})=\delta_{\textbf{m},\textbf{0}}d_{\textbf{0}}(p_1,p_2)+\lambda_1^{m_1}\lambda_2^{m_2}(X_{\textbf{p}}+X_{\textbf{m}})~~\mbox{if}~~\Delta_{\textbf{m}}=0.
\end{equation}
Then $g_{\textbf{0}}(\bm{\partial})=d_{\textbf{0}}(p_1,p_2)+X_{\textbf{p}}$.\par
Next we further calculate $g_{\textbf{0}}(\bm{\partial})$. We take $\textbf{n}=(n_1,n_2)\neq\textbf{0},\textbf{m}=-\textbf{n},\Delta_{\textbf{n}}\neq0$ in (\ref{3.4}), then $g_{\textbf{n}}(\bm{\partial}+\textbf{n})g_{-\textbf{n}}(\bm{\partial})-g_{-\textbf{n}}(\bm{\partial}-\textbf{n})g_{\textbf{n}}(\bm{\partial})=2\Delta_{\textbf{n}}g_{\textbf{0}}(\bm{\partial})$, which could imply $d_{\textbf{0}}(p_1,p_2)=0$. Hence, for any $(p_1,p_2)\neq(0,0)$, we have
\begin{equation}\label{equ3.23}
	g_{\textbf{m}}(\bm{\partial})=\lambda_1^{m_1}\lambda_2^{m_2}(X_{\textbf{p}}+X_{\textbf{m}})~~\mbox{if}~~\Delta_{\textbf{m}}=0.
\end{equation}
The proof is completed.
	\end{proof}
\end{lemm}

From Theorem \ref{th3.2} to Lemma \ref{lemm3.7}, we obtain the following theorem:
\begin{theo}
	Suppose that $V$ is an $L(p_1, p_2)$-module or an $\widetilde{L}(p_1, p_2)$-module with $(p_1, p_2)\in \mathbb{C}^{\times}\times\mathbb{C}^{\times}$ when and is free of rank one when restricted to $U(\mathfrak{h})$. Then $V\cong\Omega(\bm{\lambda},\alpha,b_{\textbf{0}})$ for some $\bm{\lambda}\in\mathbb{C}^{\times}\times\mathbb{C}^{\times},$ $\alpha\in\mathbb{C},b_{\textbf{0}}\in\mathbb{C}$ with the action given by
	\begin{equation}\label{3.29}
		E(\textbf{m})f(\partial_1,\partial_2)=\lambda_1^{m_1}\lambda_2^{m_2}(X_{\textbf{p}}-\Delta_{\textbf{m}}\alpha+X_{\textbf{m}})f(\partial_1-m_1,\partial_2-m_2),
	\end{equation}
	\begin{equation}\label{3.30}
		t^\textbf{m}f(\partial_1,\partial_2)=\lambda_1^{m_1}\lambda_2^{m_2}b_{\textbf{0}}f(\partial_1-m_1-p_1,\partial_2-m_2-p_2),
	\end{equation}
and $\partial_1,\partial_2$ act as
 left multiplication. In particular, when $b_{\textbf{0}}=0$, the above modules are $\mathcal{B}(p_1, p_2)$-modules or $\widetilde{\mathcal B}(p_1, p_2)$-modules.
\end{theo}

\subsection{The case $\textbf{p}=(0,0)$}
\
\newline
\indent
We recall the following result from \cite{Z}:
\begin{theo}\label{th3.9} For $\bm{\lambda}\in\mathbb{C}^{\times}\times\mathbb{C}^{\times}$, $\bm{\beta}\in\mathbb{C}\times\mathbb{C}$, let
	$\Omega(\bm{\lambda},\bm{\beta})=\mathbb{C}[\partial_1,\partial_2]$ be the $\mathcal L$-module given by
	\begin{equation}\label{3.38}
		E(\textbf{m})f(\partial_1,\partial_2)=\lambda_1^{m_1}\lambda_2^{m_2}(m_2(\partial_1+\beta_1)-m_1(\partial_2+\beta_2))f(\partial_1-m_1,\partial_2-m_2),
	\end{equation}
where $f(\partial_1,\partial_2)\in\mathbb{C}[\partial_1,\partial_2]$. Then any $\mathcal{L}$-module with rank one free $U(\mathfrak{h})$-action is isomorphic to some $\Omega(\bm{\lambda},\bm{\beta})$ for $\bm{\lambda}\in\mathbb{C}^{\times}\times\mathbb{C}^{\times}$ and $\bm{\beta}\in\mathbb{C}\times\mathbb{C}$.
\end{theo}
In this case, (\ref{3.5})-(\ref{3.4}) yield
\begin{equation}\label{4.5}
	h_{\textbf{n}}(\bm{\partial}-\textbf{m})h_{\textbf{m}}(\bm{\partial})-h_{\textbf{m}}(\bm{\partial}-\textbf{n})h_{\textbf{n}}(\bm{\partial})=0,
\end{equation}
\begin{equation}\label{4.6}
	h_{\textbf{n}}(\bm{\partial}-\textbf{m})g_{\textbf{m}}(\bm{\partial})-g_{\textbf{m}}(\bm{\partial}-\textbf{n})h_{\textbf{n}}(\bm{\partial})=|\textbf n, \textbf m|h_{\textbf{m}+\textbf{n}}(\bm{\partial}),
\end{equation}
\begin{equation}\label{4.4}
	g_{\textbf{n}}(\bm{\partial}-\textbf{m})g_{\textbf{m}}(\bm{\partial})-g_{\textbf{m}}(\bm{\partial}-\textbf{n})g_{\textbf{n}}(\bm{\partial})=|\textbf n, \textbf m|g_{\textbf{m}+\textbf{n}}(\bm{\partial}).
\end{equation}

	Now we calculate the expressions of $h_{\textbf{n}}(\bm{\partial})$ and $g_{\textbf{n}}(\bm{\partial})$:
	\begin{lemm} We have that $g_{\textbf{m}}(\bm{\partial})=\lambda_1^{m_1}\lambda_2^{m_2}(m_2(\partial_1+\beta_1)-m_1(\partial_2+\beta_2))$ and $h_{\textbf{m}}(\bm{\partial})=b_{\textbf{m}}\in\mathbb{C}$.
	\begin{proof}
		By Theorem \ref{th3.9}, $g_{\textbf{m}}(\bm{\partial})=\lambda_1^{m_1}\lambda_2^{m_2}(m_2(\partial_1+\beta_1)-m_1(\partial_2+\beta_2))$. Write $h_{\textbf{m}}(\bm{\partial})=\sum\limits_{k=0}^M\sum\limits_{l=0}^N a_{kl}\partial_1^k\partial_2^l$, and we claim that $M=N=0$.
	
		As before setting $\textbf{m}=\textbf{n}\neq\textbf{0}$ in (\ref{4.6}) leads to
		$\sum\limits_{k=0}^M\sum\limits_{l=0}^N a_{kl}(\partial_1-m_1)^k(\partial_2-m_2)^l=\sum\limits_{k=0}^M\sum\limits_{l=0}^N a_{kl}\partial_1^k\partial_2^l$. According to Lemma {\ref{lemm3.0}}, we have
	\begin{equation}\label{3.34}
		\left\{\begin{array}{l}h_{\textbf{m}}(\bm{\partial})\in\mathbb{C}[\partial_2]~~\mbox{if}~~m_1\neq0,m_2=0,
			\\h_{\textbf{m}}(\bm{\partial})\in\mathbb{C}[\partial_1]~~\mbox{if}~~m_1=0,m_2\neq0,
			\\h_{\textbf{m}}(\bm{\partial})=x_{\textbf{m}}+x_{\textbf{m}}^{(1)}\partial_1+x_{\textbf{m}}^{(2)}\partial_2~~\mbox{if}~~m_1\neq0,m_2\neq0\mbox{, where}~~m_1x_{\textbf{m}}^{(1)}=-m_2x_{\textbf{m}}^{(2)}.
		\end{array}\right.
	\end{equation}\par
Taking $\textbf{m}=(m_1,0),\textbf{n}=(0,n_2),m_1,n_2\in\mathbb{Z}^*$ in (\ref{4.5}), then we get $\sum\limits_{k=0}^M\sum\limits_{l=0}^N a_{k}a_l'(\partial_1-m_1)^k \partial_2^l=\sum\limits_{k=0}^M\sum\limits_{l=0}^N a_{k}a_l'\partial_1^k(\partial_2-n_2)^l$, it is easy to see $M=0$ and $N=0$.
i.e $h_{\textbf{m}}(\bm{\partial})=b_{\textbf{m}}\in\mathbb C$ for the case of either $m_1=0, m_2\neq 0$ or $m_2=0, m_1\neq 0$.

Let $\textbf{m}=(m_1,m_2),\textbf{n}=(n_1,0),m_1,m_2,n_1\neq0$ in (\ref{4.5}), we have $x_{\textbf{m}}+x_{\textbf{m}}^{(1)}\partial_1+x_{\textbf{m}}^{(2)}\partial_2=x_{\textbf{m}}+x_{\textbf{m}}^{(1)}(\partial_1-n_1)+x_{\textbf{m}}^{(2)}\partial_2$, then we get $x_{\textbf{m}}^{(1)}=0$. Similarly, we get $x_{\textbf{m}}^{(2)}=0$. So we could denote
\begin{center}
	$h_{\textbf{m}}(\bm{\partial})=b_{\textbf{m}}\in\mathbb{C}$ if $\textbf{m}=(m_1,m_2),m_1\neq0,m_2\neq0.$
\end{center}\par
We set $\textbf{n}=\textbf{0},\textbf{m}\neq\textbf{0}$ in (\ref{4.6}), then $h_{\textbf{0}}(\bm{\partial})=h_{\textbf{0}}(\bm{\partial}-\textbf{m})$. Then we can obtain $h_{\textbf{0}}(\bm{\partial})=b_{\textbf{0}}\in\mathbb{C}$ if we take $\textbf{m}=(m_1,0),m_1\neq0$ and $\textbf{m}=(0,m_2),m_2\neq0$, respectively.\par
Overall, we can see that $h_{\textbf{m}}(\bm{\partial})=b_{\textbf{m}}\in\mathbb{C}$ for all $\textbf{m}\in\mathbb{Z}\times\mathbb{Z}$.
\end{proof}
\end{lemm}

\begin{lemm}
	For any $(0,0)\neq \textbf{m}\in\mathbb{Z}\times\mathbb{Z}$, there exists $k\in\mathbb{C}$ such that $b_{\textbf{m}}=\lambda_1^{m_1}\lambda_2^{m_2}k$.
	\begin{proof}
		Suppose $(\textbf{m}|\textbf{n}^{\bot})\neq0$ in (\ref{4.6}), then we get
			$\lambda_1^{m_1}\lambda_2^{m_2}b_{\textbf{n}}=b_{\textbf{m}+\textbf{n}}$.

	If $m_1\neq0$, we have $b_{(m_1,m_2)}=\lambda_1^{m_1}\lambda_2^{m_2-1}b_{(0,1)}$; if $m_2\neq0$, we have $b_{(m_1,m_2)}=\lambda_1^{m_1-1}\lambda_2^{m_2}b_{(1,0)}$. Write $k=\lambda_2^{-1}b_{(0,1)}=\lambda_1^{-1}b_{(1,0)}$, then
	\begin{center}
		$b_\textbf{m}=\lambda_1^{m_1}\lambda_2^{m_2}k$ if $\textbf{m}\neq\textbf{0}.$
	\end{center}
This completes the proof.
	\end{proof}
\end{lemm}
In summary, we have the following result.
\begin{theo}
	Suppose $V$ is an $L(0,0)$-module or an $\widetilde{L}(0,0)$-module and is free of rank one
when restricted to $U(\mathfrak{h})$, then $V\cong\Omega(\bm{\lambda},\bm{\beta},b_{\textbf{0}},k)$ for some $\bm{\lambda}\in\mathbb{C}^{\times}\times\mathbb{C}^{\times},$ $\bm{\beta}\in\mathbb{C}\times\mathbb{C},b_{\textbf{0}},k\in\mathbb{C}$ with the module structure given by:
	\begin{equation}
		E(\textbf{m})f(\partial_1,\partial_2)=\lambda_1^{m_1}\lambda_2^{m_2}(m_2(\partial_1+\beta_1)-m_1(\partial_2+\beta_2))f(\partial_1-m_1,\partial_2-m_2),
	\end{equation}
	\begin{equation}
			t^\textbf{m}f(\partial_1,\partial_2)=
	\left\{\begin{array}{l}\lambda_1^{m_1}\lambda_2^{m_2}kf(\partial_1-m_1,\partial_2-m_2),~~\mbox{if}~~\textbf{m}\neq\textbf{0},
		\\b_{\textbf{0}}f(\partial_1,\partial_2),~~\mbox{if}~~\textbf{m}=\textbf{0},
	\end{array}\right.
	\end{equation}
and $\partial_1,\partial_2$ act as left multiplication.
\end{theo}\par
Consequently, we obtain the main result of this paper:
\begin{theo}
	Let $\textbf{p}\in\mathbb{C}\times\mathbb{C}$. Suppose $V$ is an $L(\textbf{p})$-module or an $\widetilde{L}(\textbf{p})$-module with free rank one action of the Cartan subalgebra
	$U(\mathfrak{h})$.
	
	$(1)$ If $\textbf{p}\neq(0,0)$, then $V\cong\Omega(\bm{\lambda},\alpha,b_{\textbf{0}})$ for some $\bm{\lambda}\in\mathbb{C}^{\times}\times\mathbb{C}^{\times},$ $\alpha\in\mathbb{C},b_{\textbf{0}}\in\mathbb{C}$ with
	the module action: 
	\begin{equation}\label{3.29}
		E(\textbf{m})f(\partial_1,\partial_2)=\lambda_1^{m_1}\lambda_2^{m_2}[(m_2+p_2)(\partial_1+p_1\alpha)-(m_1+p_1)(\partial_2+p_2\alpha)]f(\bm{\partial}-\textbf{m}),
	\end{equation}
	\begin{equation}\label{3.30}
		t^\textbf{m}f(\bm{\partial})=
		\lambda_1^{m_1}\lambda_2^{m_2}b_{\textbf{0}}f(\bm{\partial}-\textbf{m}-\textbf{p}),
	\end{equation}
	and $\partial_1,\partial_2$ act as left multiplication. In particular, when $b_{\textbf{0}}=0$, the above module is a $\mathcal{B}(\textbf{p})$-module or $\widetilde{\mathcal B}(\textbf{p})$-module.\par
	$(2)$ If $\textbf{p}=(0,0)$, then $V\cong\Omega(\bm{\lambda},\bm{\beta},b_{\textbf{0}},k)$ for some $\bm{\lambda}\in\mathbb{C}^{\times}\times\mathbb{C}^{\times},$ $\bm{\beta}\in\mathbb{C}\times\mathbb{C},b_{\textbf{0}},k\in\mathbb{C}$ with the module action:
	\begin{equation}\label{3.29}
		E(\textbf{m})f(\bm{\partial})=\lambda_1^{m_1}\lambda_2^{m_2}(m_2(\partial_1+\beta_1)-m_1(\partial_2+\beta_2))f(\bm{\partial}-\textbf{m}),
	\end{equation}
	\begin{equation}\label{3.30}
		t^\textbf{m}f(\bm{\partial})=
		\left\{\begin{array}{l}\lambda_1^{m_1}\lambda_2^{m_2}kf(\bm{\partial}-\textbf{m}),~~\mbox{if}~~\textbf{m}\neq\textbf{0},
			\\b_{\textbf{0}}f(\bm{\partial}),~~\mbox{if}~~\textbf{m}=\textbf{0},
		\end{array}\right.
	\end{equation}
	where $\partial_1,\partial_2$ act as left multiplication. In particular, when $b_{\textbf{0}}=k=0$, the above module is an $\mathcal{L}$-module.
\end{theo}

\vskip30pt
\noindent {\bf Data availability statement}: All data are included in the manuscript.

	\vskip30pt \centerline{\bf ACKNOWLEDGMENT}

The research is partially supported by the Simons Foundation MP-TSM-00002518,
NSFC grants (Nos. 12171303; 12071276, 11931009, 12226402; 12271332) and NSF of Shanghai Municipality grant no. 22ZR142 4600.

\end{document}